\documentclass[a4paper,reqno,10pt]{amsart}
\usepackage{amsfonts,amssymb,amsthm,amsbsy,latexsym,amscd,amsmath,euscript,enumerate,manfnt, marvosym,verbatim,calc,mathrsfs,marvosym,tensor,textcomp ,color,xcolor,tikz,setspace,upgreek,bbm,bm}

\definecolor{bulgarianrose}{rgb}{0.28, 0.02, 0.03}
\usepackage[linkcolor=bulgarianrose,citecolor=bulgarianrose,colorlinks=true,hypertexnames=false]{hyperref}

\newtheorem{theorem}{Theorem}[section]
\newtheorem{corollary}[theorem]{Corollary}
\newtheorem{lemma}[theorem]{Lemma}

\newtheorem{proposition}[theorem]{Proposition}

\theoremstyle{definition}
\newtheorem*{definition}{Definition}
\newtheorem{construction}[theorem]{Construction}
\newtheorem*{question}{Question}
\newtheorem*{remark}{Remark}

\newtheorem*{acknowledgement}{Acknowledgement}

\newcommand{\M}[2][]{\ensuremath{{\textnormal{M}}_{#1}(#2)}}
\newcommand{\K}[2][]{\ensuremath{{\textnormal{K}}_{#1}(#2)}}
\newcommand{\N}[1][]{\ensuremath{{\textnormal{N}}_{#1}}}
\newcommand{\di}[2][]{\ensuremath{{\textnormal{d}}_{#1}(#2)}}
\newcommand{\ed}[2][]{\ensuremath{{\textnormal{e}}_{#1}(#2)}}

\makeatletter
\def\imod#1{\allowbreak\mkern10mu({\operator@font mod}\,\,#1)}
\def\@textbottom{\vskip\z@\@plus 18pt}
\let\@texttop\relax
\makeatother

\title[Universality and Extremality]{\textnormal{On the Universality and Extremality of graphs\\ with a distance constrained colouring}}
\author{Kaushik Majumder}
\author{Ushnish Sarkar}
\address{\newline R C Bose Centre for Cryptology and Security\\ \newline Indian Statistical Institute\\ \newline $202$ Barrackpore Trunk Road\\\newline Kolkata - $700108$, India.\newline \textnormal{\textestimated-Mail}: {\tt kaushikbnmajumder\MVAt gmail.com}}

\address{\newline Department of Science and Humanities\\ \newline Sidhu Kanhu Birsa Polytechnic, Keshiary\\ \newline Paschim Medinipur, West Bengal - $721133$, India \newline \textnormal{\textestimated-Mail}: {\tt usn.prl\MVAt gmail.com}}

\subjclass[2010]{Primary: 05C15, 05C35, 05C75, 05C78. Secondary: 05D99.}
\keywords{Frequency assignment, Vertex labelling at distance two, $L(2,1)-$Colouring, Lambda number.}

\begin{document}
\begin{abstract}
A lambda colouring (or $L(2,1)-$colouring) of a graph is an assignment of non-negative integers (with minimum assignment $0$) to its vertices such that the adjacent vertices must receive integers at least two apart and 
vertices at distance two must receive distinct integers. The lambda chromatic number (or the $\lambda$ number) of a graph $G$ is the least positive integer among all the maximum assigned positive integer over all possible lambda colouring of the graph $G$. Here we have primarily shown that every graph with lambda chromatic number $t$ can be embedded in a graph, with lambda chromatic number $t$, which admits a partition of the vertex set into colour 
classes of equal size. It is further proved that if an $n-$vertex graph with lambda chromatic number $t\geq5$, 
where $n\geq t+1$, contains maximum number of edges, then the vertex set of such graph admits an equitable 
partition. For such an admitted equitable partition there are either $0$ or $\min\{|A|,|B|\}$ number of edges 
between each pair $(A,B)$ of subsets (i.e. roughly, such partition is a ``sparse like'' equitable partition). 
Here we establish a classification result, identifying all possible $n-$vertex graphs with lambda chromatic 
number $t\geq3$, where $n\geq t+1$, which contain maximum number of edges. Such classification provides a 
solution of a problem posed more than two decades ago by John P. Georges and David W. Mauro.
\end{abstract}

\maketitle

\section{Introduction}

The \emph{channel assignment} problem is the task of assigning frequencies to radio transmitters  
of a communication network. In this problem, there is a trade off between deploying minimum number of 
frequencies (or channels) and avoiding interference due to proximity of transmitters. This distance restriction 
necessitates a separation of frequencies among nearby transmitters in order to mitigate the interference. 
The usable spectrum of frequencies is a scarce and a costly resource. 
For this reason, an efficient assignment of frequencies is desirable.
The frequency assignment to the transmitters, constrained by distance related parameters, can be mapped to  
varieties of \emph{distance constrained colouring} problems of a graph. Hale \cite{Hale} modelled 
these problems as several generalised versions of graph colouring problem. In one of such models, 
the transmitters are considered as vertices and edges correspond to the unordered pairs of interfering transmitters. 
The assignment of frequencies (represented by non-negative integers) is done in such a manner that ``close'' transmitters (i.e. vertices at distance two) are assigned different frequencies and ``very close'' 
transmitters (i.e. adjacent vertices) are assigned frequencies in a difference of at least two. This channel assignment problem is translated to a colouring (or labelling) problem of graphs. Griggs and Yeh 
\cite{MR1186826} had referred this colouring problem as $L(2,1)-$colouring problem of graphs. We refer it as 
lambda colouring problem of graphs. From the complexity point of view, the lambda colouring problem of 
graphs is an $\mathcal{NP-}$hard problem \cite[Theorem~57]{Hale}. Survey articles on this well investigated 
problem can be found in \cite{Calamoneri,MR1139583,MR2245647}.

Throughout this article the set of all non-negative integers is denoted by $\mathbb{N}$. All the graphs are 
simple and their vertex sets are non-empty and finite. For a graph $G$, the vertex set and the edge set 
are denoted by $V(G)$ and $E(G)$ respectively. The subset $\N[u]=\left\{v\in V(G):\{u,v\}\in E(G)\right\}$ of the vertex 
set of a graph $G$ is called the \emph{neighbour set} or simply the \emph{neighbour} of the vertex $u$. 
A (vertex) \emph{colouring} of a graph $G$ is a mapping $c:V(G)\rightarrow\mathbb{N}$. The 
\emph{lambda colouring} of a graph $G$ is a mapping $c:V(G)\rightarrow\mathbb{N}$ such that for each $u,v\in V(G)$, 
$|c(u)-c(v)|+\di[G]{u,v}\geq3$. Here $\di[G]{u,v}$ denotes the \emph{distance} between vertices $u$ and $v$, 
i.e. the minimum number of edges connecting the vertices $u$ and $v$ through a path. If there is no 
edges connecting the vertices $u$ and $v$ through a path, then we may put $\di[G]{u,v}=\infty$. By well 
ordering property of $\mathbb{N}$, the range of the mapping $c$ attains $\underset{u\in V(G)}{\min} c(u)$. 
Therefore without loss of generality, we assume $\underset{u\in V(G)}{\min} c(u)=0$. The 
\emph{lambda chromatic number} of the graph $G$ is  the positive integer  
$\min\left\{\underset{u\in V(G)}{\max}c(u): c\textup{ is a lambda colouring of } G\right\}$.
A lambda colouring $c$ of $G$ is said to be \emph{optimal} if $\underset{u\in V(G)}{\max}c(u)$ equals to the 
lambda chromatic number. The \emph{coloured partition} of the vertex set $V(G)$ with respect to a lambda colouring 
$c$ is $C_{0},C_{1},\ldots,C_{m},\ldots,C_{T}$, where $C_{m}:=\{u\in V(G):c(u)=m\}$ and $T=\underset{u\in V(G)}{\max}c(u)$. 
Such coloured partition is said to be \emph{equitable} if for all integers $i$ and $j$ with $0\leq i,j\leq T$, 
$||C_{i}|-|C_{j}||\leq1$. With respect to a lambda colouring $c$ of the graph $G$, if the coloured partition 
of $V(G)$ is an equitable partition, then we call such lambda colouring is \emph{equitable}. Fu and Xie 
\cite{MR2598700} have studied equitable lambda colouring for Sierpi\'{n}ski graphs. A positive integer 
$h$ is said to be a \emph{hole} of the lambda colouring $c$, if for each vertex $u$, $c(u)\neq h$ but there 
exist at least one vertex $v$ such that $h<c(v)$. Note that a hole corresponds an empty colour class.
(Caution: The lambda colouring may not be an onto mapping, contrary to usual colouring of graphs.) 

Fishburn and Roberts \cite{MR1999706,MR2257271} extensively studied the possible graphs admitting at 
least one optimal but onto lambda colouring. Such graphs are known as \emph{full colourable} graphs. 
However in \cite{MR2332322}, the authors concentrated specifically on onto lambda colourings of any graph 
irrespective of being full colourable. They studied the associated optimal value in terms of its bounds. In fact,
they had obtained their results in the case where these bounds are attained. These results were expressed in 
terms of the number of edges, diameters and number of connected components. On the other hand, the extremal 
nature of the graphs, from the view point of inter-relation between lambda chromatic number and minimum number of holes, was studied in \cite{MR2354760}. 

In this article, we concentrated mainly on the edge distribution of an $n-$vertex graph (i.e. a graph with 
$n$ number of vertices) with lambda chromatic number $t$. In fact, this article has a two-way orientation 
namely, universality and extremality of family of graphs with lambda chromatic number $t$. 
\begin{question}
Our main focus is to answer the following two questions.  
\begin{enumerate}[(a)]
\item Can we find a graph $\Omega$ with lambda chromatic number $t$ such that any graph $G$ with lambda chromatic number 
$t$ is a subgraph of $\Omega$?
\item By means of explicit construction, can we classify all the $n-$vertex graphs, with lambda chromatic number $t$, which contain maximum number of edges? 
\end{enumerate} 
\end{question}
Regarding the Question (b), for $n\leq t+1$ we refer \cite{MR1383991}. However, a reader may realise that the answer of this case is incurred within Proposition~\ref{base}. Here we focus mainly for $n\geq t+1$. 

\section{Two examples and their universal properties}

If $G$ is a graph with lambda chromatic number $2$, then $G$ is a disjoint union of some edges. But if $G$ is a graph with lambda chromatic number $t\geq3$, then the problems (described in Question (a) and (b)) become non-trivial. Here we begin this section with (a) a sequence of graphs $\{\mathbb{G}_{n}\}_{n=3}^{\infty}$ and (b) a doubly sequence of family (set) of graphs $\{\mathsf{G}(t,l): t\geq3,l\geq1\}$ and study their lambda chromatic number and other properties. 

\begin{construction}\label{G_n}
Let $\{\mathbb{G}_{n}:n\in\mathbb{N}, n\geq3\}$ be a sequence of simple graphs defined via a recursive rule as follows: $\mathbb{G}_{3}$ be the graph with edge set $\left\{\{v_{0},v_{2}\},\{v_{0},v_{3}\},\{v_{1},v_{3}\}\right\}$. For $n\geq4$, 
the graph $\mathbb{G}_{n}$ has vertex set $\{v_{i}:0\leq i\leq n\}$. The edges of $\mathbb{G}_{n}$ are all the edges of $\mathbb{G}_{n-1}$ and the edge of the form $\{v_{i},v_{n}\}$, where $i$ is an integer with $0\leq i\leq n-2$. In total, for each integer $t\geq3$, $\mathbb{G}_{t}$ has exactly $(t+1)$ vertices and $\binom{t}{2}$ edges.
\end{construction}

It will be shown later that each graph can be modified through edge standardisation into disjoint union of some $\mathbb{G}_{t}$'s with or without some deleted vertices.  

\begin{construction}\label{G(t,l)}
Let $t\geq3$ be an integer and $V_{i}$, where $i$ is an integer with $0\leq i\leq t$,  be mutually disjoint sets of size $l$. Let $\mathsf{G}(t,l)$ be the family of graphs with vertex set $\overset{t}{\underset{i=0}\sqcup}V_{i}$. Each graph  $G\in\mathsf{G}(t,l)$ satisfies the following two properties. 
\begin{itemize}
\item If $x\in V_{m}$, then there exists a unique $y\in V_{p}$ such that $\{x,y\}$ is an edge of $G$, where $m$ 
and $p$ are integers with $0\leq m\leq p-2\leq t-2$.
\item Whenever $u$, $v\in V_{m}$, where $0\leq m\leq t$, $\{u,v\}$ is not an edge of $G$.
\end{itemize}
In total, each $G\in\mathsf{G}(t,l)$ has exactly $(t+1)l$ vertices and $\binom{t}{2}l$ edges.
\end{construction}

There is a link between the above two constructions. Precisely for each $n\geq3$, the graph $\mathbb{G}_{n}$  is the only member of $\mathsf{G}(n,1)$. In the following two results, we obtain an optimal lambda colouring of $\mathbb{G}_{n}$, where $n\geq3$, namely $v_{i}\mapsto i$ from $V(\mathbb{G}_{n})=\{v_{i}:0\leq i\leq n\}$ to $\mathbb{N}$. We found that such optimal colouring is an onto mapping. It concludes that under such colouring  
of $\mathbb{G}_{n}$ there is no hole, for each $n\geq3$.

\begin{lemma}\label{distance}
For each $n\geq4$ and $u,v\in V(\mathbb{G}_{n})$, $1\leq\di[\mathbb{G}_{n}]{u,v}\leq2$.
\end{lemma}
\begin{proof}
We compare the graphs $\mathbb{G}_{m}$ and $\mathbb{G}_{m+1}$, where $m\geq4$. We note that  $\mathbb{G}_{m}$ is a subgraph of $\mathbb{G}_{m+1}$. Therefore $\di[\mathbb{G}_{m+1}]{u,v}\leq\di[\mathbb{G}_{m}]{u,v}$, for each $u,v\in V(\mathbb{G}_{m})$. Apart form all the vertices and edges of $\mathbb{G}_{m}$, in $\mathbb{G}_{m+1}$ the new vertex is $v_{m+1}$ and the new edges are edge of the form  $\{v_{i},v_{m+1}\}$, where $i$ is an integer with $0\leq i\leq m-1$. Hence 
\begin{align*}
\di[\mathbb{G}_{m+1}]{v_{m+1},u}=\left\{\begin{array}{lcr}
                1 & \textnormal{if} & u\neq v_{m}\\
                2 & \textnormal{if} & u=v_{m}.
            \end{array}\right.
\end{align*}
Therefore supposing the result is true for $n=m\geq4$, we conclude that the result is true for $n=m+1$. The result is true for $n=4$. Hence the result follows by induction on $n$.
\end{proof}

\begin{theorem}\label{lambda_G_n}
For each $n\geq3$,the mapping $v_{i}\mapsto i$, from $V(\mathbb{G}_{n})$ to $\mathbb{N}$, is a lambda colouring and the lambda chromatic number of $\mathbb{G}_{n}$ is $n$.
\end{theorem}
\begin{proof}
Suppose the result is true for $n=m\geq3$. Therefore $\bar{c}:V(\mathbb{G}_{m})\rightarrow\{0,1,\ldots,m\}$ defined by 
$\bar{c}(v_{i})=i$ is a lambda colouring, then the mapping $c:V(\mathbb{G}_{m+1})\rightarrow\mathbb{N}$ defined by
\begin{align*}
c(v_{i})=\left\{\begin{array}{lcr}
                \bar{c}(v_{i}) & \textnormal{if} & 0\leq i\leq m\\
                m+1 & \textnormal{if} & i=m+1
            \end{array}\right.
\end{align*}
is a colouring of $\mathbb{G}_{m+1}$. 

\noindent{\textsl{Claim} :} $c$ is a lambda colouring of $\mathbb{G}_{m+1}$.
\begin{proof}[\tt{Proof of claim} :]\renewcommand{\qedsymbol}{}
We note that for each $u$, $v\in V(\mathbb{G}_{m+1})$ $|c(u)-c(v)|\geq1$, so if $\di[\mathbb{G}_{m+1}]{u,v}=2$, then  
$|c(u)-c(v)|+\di[\mathbb{G}_{m+1}]{u,v}\geq3$. Now if $\di[\mathbb{G}_{m+1}]{u,v}=1$ for some $u$, $v\in V(\mathbb{G}_{m})$, then 
\begin{equation*}
|c(u)-c(v)|=|\bar{c}(u)-\bar{c}(v)|\geq2,
\end{equation*}
since by assumption, $\bar{c}$ is a lambda colouring of $\mathbb{G}_{m}$. So by Lemma~\ref{distance} the only case left where $\di[\mathbb{G}_{m+1}]{v_{m+1},u}=1$ with $u\in V(\mathbb{G}_{m})$. Here we explicitly have $u=v_{i}$, where $0\leq i\leq m-1$ and consequently $\bar{c}(u)=\bar{c}(v_{i})=i$ It implies that
\begin{equation*}
|c(v_{m+1})-c(u)|=|m+1-\bar{c}(u)|=|m+1-i|\geq2.
\end{equation*}
Hence the claim is established.
\end{proof}

By the above claim, it implies that the lambda chromatic number of $\mathbb{G}_{m+1}$ is at most $m+1$. Let $c:V(\mathbb{G}_{m+1})\rightarrow\mathbb{N}$ be a lambda colouring. By Lemma~\ref{distance} for each $u,v\in V(\mathbb{G}_{m+1})$, $1\leq \di[\mathbb{G}_{m+1}]{u,v}\leq2$, therefore every vertex must receive distinct colours 
(non-negative integers). Hence we need at least $|V(\mathbb{G}_{m+1})|=m+2$ colours to colour the vertices of $\mathbb{G}_{m+1}$. Consequently,
\begin{equation*}
\max\{c(u):u\in V(\mathbb{G}_{m+1})\}\geq\max\{0,1,\ldots,m+1\}=m+1.
\end{equation*}
Since $c:V(\mathbb{G}_{m+1})\rightarrow\mathbb{N}$ is an arbitrary lambda colouring, the lambda chromatic number of $\mathbb{G}_{m+1}$ is at least $m+1$. This means the result is true for $n=m+1$. It can be verified directly that the 
result is true for $n=3$. Hence the result follows by induction on $n$.
\end{proof}

The following two results give us a no-hole optimal lambda colouring of each of the member graphs of 
$\mathsf{G}(t,l)$.

\begin{theorem}\label{lambdacolour}
For positive integers $t\geq3$ and $l$, the colouring map $v_{m}\mapsto m$, where $v_{m}\in V_{m}$ and $0\leq m\leq t$ is a lambda colouring of each member graph $G\in\mathsf{G}(t,l)$. 
\end{theorem}
\begin{proof}
We denote the aforementioned mapping by $c:\overset{t}{\underset{i=0}\sqcup}V_{i}\rightarrow\mathbb{N}$. Let $G$ be a member graph of $\mathsf{G}(t,l)$ and $u$,$v$ are two vertices of $G$. If $\di[G]{u,v}\geq3$, then 
$|c(u)-c(v)|+\di[G]{u,v}\geq3$. Therefore we have to check only the following two cases.

\noindent{\textsf{Case} I : }  $\di[G]{u,v}=2$. To show $|c(u)-c(v)|\geq1$.\\
Suppose $c(u)=c(v)$. Since $\di[G]{u,v}=2$, it means there exists $w\in V_{c(w)}$ such that $\{u,w\}$ and $\{v,w\}$ are edges of $G$. A contradiction to the definition of $G$ as $u,v\in V_{c(u)}$. Hence $|c(u)-c(v)|\geq1$.

\noindent{\textsf{Case} II : }  $\di[G]{u,v}=1$, i.e. if $\{u,v\}$ is an edge of $G$. To show $|c(u)-c(v)|\geq2$.\\
Let $u\in V_{m}$, where $1\leq m\leq t-1$, then $v\in V_{i}$, where $i$ is an integer with $0\leq i\leq t$ but 
$i\neq m-1,m,m+1$. Hence $|c(u)-c(v)|=|m-i|\geq2$. If $u\in V_{0}$, then $v\in V_{i}$, where $i$ is an integer with $2\leq i\leq t$. Hence $|c(u)-c(v)|=|0-i|\geq2$. Also if $u\in V_{t}$, then $v\in V_{i}$, where $i$ is an integer with $0\leq i\leq t-2$. Hence $|c(u)-c(v)|=|t-i|\geq2$.
\end{proof}

\begin{theorem}
For positive integers $t\geq3$ and $l$, the lambda chromatic number of each member graph $G\in\mathsf{G}(t,l)$ is $t$.
\end{theorem}
\begin{proof}
Let $c:V(G)\rightarrow\mathbb{N}$ be a lambda colouring of $G$. Fix an $v_{0}\in V_{0}$, and $v_{0i}$, where $2\leq i\leq t$, denotes the unique neighbour of $v_{0}$ in $V_{i}$. Hence $v_{0}$ has $t-1$ neighbours in $G$. Since 
$\di[G]{v_{0i},v_{0j}}\leq2$, for $2\leq i,j\leq t$, each of the $v_{0i}$ must receive distinct non-negative integers (colours) namely $c(v_{02}),\ldots,c(v_{0t})$. 
Moreover, $v_{0}$ is adjacent to each of $v_{0i}$, therefore $|c(v_{0})-c(v_{0i})|\geq2$, for each integer $i$ with 
$2\leq i\leq t$. It implies that 
\begin{equation*}
\max\{c(v):v\in V(G)\}\geq\max\{c(v_{0}),c(v_{02}),\ldots,c(v_{0t})\}\geq t.
\end{equation*}
Since $c$ is an arbitrary lambda colouring, we have the lambda chromatic number of such graph $G$ is at least $t$. From Theorem~\ref{lambdacolour}, we have the lambda chromatic number of such graph $G$ is at most $t$ and the result follows.
\end{proof}

Let $G$ be a graph and $c:V(G)\rightarrow\mathbb{N}$ be a lambda colouring. We note that the relation $\sim$ defined on $V(G)$ by $u\sim v$ if and only if $c(u)=c(v)$. Such relation is an equivalence relation on $V(G)$. By the notation $[c(u)]$, we denote the equivalence class containing the vertex $u$. Note that $[c(u)]$ is non-empty and equals to $C_{m}$ for some integer $m$, with $0\leq m\leq t$.

\begin{lemma}\label{colourdistribution}
Let $G$ be a graph, $c:V(G)\rightarrow\mathbb{N}$ be a lambda colouring and $[c(u)]$, $[c(v)]$ be two different colour classes. Then for each $x\in[c(u)]$ there exists at most one vertex $y\in[c(v)]$ such that $\{x,y\}$ is an edge of $G$.
\end{lemma}
\begin{proof}
Suppose there exist $y_{1},y_{2}\in[c(v)]$ such that $\{x,y_{1}\}$ and $\{x,y_{2}\}$ are edges. Since $c(y_{1})=c(y_{2})=c(v)$, we have $y_{1}$ and $y_{2}$ are not adjacent, hence $\di[G]{y_{1},y_{2}}=2$. A contradiction arises since \\ $2=|c(y_{1})-c(y_{2})|+\di[G]{y_{1},y_{2}}\geq3$.
\end{proof}
 
The above lemma implies that between any two distinct colour classes $A$ and $B$, the subgraph 
$\left\{\{x,z\},\{y,z\}\right\}$, where $x,y\in A$ and $z\in B$, is the forbidden subgraph in the graph $G$. The following result is one of the main theorems of this article. It asserts the affirmative answer of the Question~(a) posed in the introduction.

\begin{theorem}\label{universality}
Let $G$ be a graph with lambda chromatic number $t\geq3$ and $C_{0},\ldots,C_{t}$ be the coloured partition  of the vertex set of $G$ with respect to an optimal lambda colouring $c:V(G)\rightarrow\{0,\ldots,t\}$. Then there exists a graph $G^{*}\in\mathsf{G}(t,l)$ such that $G$ is a subgraph of $G^{*}$, where $l=\underset{u\in V(G)}{\max}{|[c(u)]|}$. 
\end{theorem}
\begin{proof}
We note that $C_{m}$ is empty for some integer $m$ with $0\leq m\leq t$ if and only if $m$ is an hole of the lambda colouring $c$, hence $0\leq|C_{m}|\leq l$. For each integer $m$ with $0\leq m\leq l$, we adjoin a disjoint set $Y_{m}$ with $C_{m}$ such that $|C_{m}|+|Y_{m}|=l$. (Here such $Y_{m}$'s, $0\leq m\leq l$, are also mutually disjoint.) So let $Y_{m}:=\{x^{m}_{i}:|C_{m}|+1\leq i\leq l\}$ and $V_{m}:=C_{m}\sqcup Y_{m}$. For 
$0\leq m,p\leq t$ and $m\neq p$, let 
\begin{equation*}
Z_{m,p}=\left\{u\in V_{m}: u\textup{ has no neighbour in } V_{p}\right\}.
\end{equation*}
By using Lemma~\ref{colourdistribution}, we have for $0\leq m\leq p-2\leq t-2$ and for each $x\in V_{m}\smallsetminus Z_{m,p}$ there exists exactly one vertex $y\in V_{p}\smallsetminus Z_{p,m}$ such that $\{x,y\}$ is an edge in $G$. Conversely, if $\{u,v\}$ is an edge in $G$, then due to the fact $c$ is a lambda colouring, there exist integers $m$ and $p$, with 
$0\leq m\leq p-2\leq t-2$, such that $[c(u)]=C_{m}$ and $[c(v)]=C_{p}$. From Lemma~\ref{colourdistribution}, we conclude that $v$ is the only neighbour of $u$ in $C_{p}$ and vice-versa.    

We construct a graph $G^{*}$ with vertex set $\overset{t}{\underset{m=0}\sqcup}V_{m}$. The edges of $G$ are edges of $G^{*}$. We note that $|Z_{m,p}|=|Z_{p,m}|$ and for each $u\in Z_{m,p}$ we associate a unique $v\in Z_{p,m}$ and construct an edge $\{u,v\}$ of $G^{*}$, where  $0\leq m\leq p-2\leq t-2$. Hence $G^{*}$ is the required member of $\mathsf{G}(t,l)$. 
\end{proof}

One of the most discussed conjectures relates the lambda chromatic number $t$ of a 
graph $G$ and $\triangle=\max\{|\N[u]|:u\in V(G)\}$. Such conjecture is made it known as the Griggs-Yeh Conjecture 
\cite[Conjecture~10.1]{MR1186826}, it states that $t\leq\triangle^{2}$. So far the best known upper bound of the lambda chromatic number is $\triangle^{2}+\triangle-2$ \cite{MR2392058}.
As a corollary of Theorem~\ref{universality}, we have proved a tight lower bound of the lambda chromatic number 
of $G$ in terms of $\triangle$. The following lower bound is tight in the sense that $\mathbb{G}_{t}$ and each 
member graph of $\mathsf{G}(t,l)$, where $l$ is a positive integer, attains such lower bound.  

\begin{corollary}
Let $G$ be a graph with lambda chromatic number $t$ and  $\triangle=\max\{|\N[u]|:u\in V(G)\}$, where 
$\N[u]$ denotes the neighbour of the vertex $u$. Then $\triangle+1\leq t$.
\end{corollary}
\begin{proof}
Let $u$ be a vertex of $G$ with $\triangle=|\N[u]|$. Then using Theorem~\ref{universality}, we have an integer $l$ (as prescribed in the theorem) and a graph $G^{*}\in\mathsf{G}(t,l)$ such that $\left\{\{u,v\}:v\in \N[u]\right\}$ is a subgraph of $G^{*}$. The result follows since $\max\{|\N[x]|:x\in V(G^{*})\}=t-1$. 
\end{proof}

\section{Maximum number of edges and equitable partition}

Let $G$ be a graph with lambda chromatic number $t$ and $c:V(G)\rightarrow\{0,\ldots,t\}$ be a lambda colouring.
Also let $C_{m}:=\{u\in V(G):c(u)=m\}:=\{u^{m}_{i}:1\leq i\leq|C_{m}|\}$ (say), where $m$ is an integer, 
$0\leq m\leq t$. Throughout this section, we assume that the aforementioned lambda chromatic number of $G$ is 
$t\geq3$. With respect to the lambda colouring $c$ of the graph $G$, we denote
\begin{equation*}
\mathfrak{M}_{c}(G):=\left\{C_{M}:|C_{M}|=\underset{0\leq j\leq t}{\max}|C_{j}|\right\},
\mathfrak{m}_{c}(G):=\left\{C_{m}:|C_{m}|=\underset{0\leq j\leq t}{\min}|C_{j}|\right\}\textup{ and }
\nabla_{c}(G):=\underset{0\leq j\leq t}{\max}|C_{j}|-\underset{0\leq j\leq t}{\min}|C_{j}|.
\end{equation*}

We also fix two more notations here. Let  
\begin{equation*}
\M{C_{0},\ldots,C_{t}}:=\overset{t-2}{\underset{i=0}\sum}\overset{t}{\underset{j=i+2}\sum}\min\{|C_{i}|,|C_{j}|\},
\end{equation*}
where $G$ is a graph with lambda chromatic number $t$ and $C_{0},\ldots,C_{t}$ be the coloured partition of the vertex set of $G$ with respect to an optimal lambda colouring $c:V(G)\rightarrow\{0,\ldots,t\}$. Suppose $X$ and $Y$ are two subsets of the vertex set of the graph $G$. The number of edges of the form $\{x,y\}$, where $x\in X$ and $y\in Y$, is denoted by 
$\ed[G]{X,Y}$.  

\begin{theorem}\label{notation_M}
Let $G$ be a graph with lambda chromatic number $t$ and $C_{0},\ldots,C_{t}$ be the coloured partition of the vertex set of 
$G$ with respect to an optimal lambda colouring $c$. Then 
\begin{enumerate}[\normalfont(a)]
\item $G$ has at most $\M{C_{0},\ldots,C_{t}}$ edges.
\item $G$ has exactly $\M{C_{0},\ldots,C_{t}}$ edges if and only if for each $x\in C_{i}$ there exists exactly one vertex $y\in C_{j}$ such that $\{x,y\}$ is an edge of $G$, where $0<|C_{i}|\leq|C_{j}|$ with $0\leq i,j\leq t$ and $|i-j|\geq2$.
\end{enumerate}
\end{theorem}
\begin{proof}
Since $c$ is a lambda colouring, we have $\ed[G]{C_{p}, C_{p+1}}=0$ for each integer $p$, with $0\leq p\leq t-1$. By using Lemma~\ref{colourdistribution}, we have $\ed[G]{C_{i},C_{j}}\leq\min\{|C_{i}|,|C_{j}|\}$, where $i$ and $j$ are integers with 
$0\leq i\leq j-2\leq t-2$. Hence the result (a) follows

If for all integers $p$ and $q$, with $0\leq p\leq q-2\leq t-2$, and for each $x\in C_{p}$ there exists exactly one vertex 
$y\in C_{q}$ such that $\{x,y\}$ is an edge of $G$, then $0<|C_{p}|\leq|C_{q}|$ and 
$\ed[G]{C_{p},C_{q}}=|C_{p}|=\min\{|C_{p}|,|C_{q}|\}$. Hence the number of edges is $\M{C_{0},\ldots,C_{t}}$. Conversely, if for some integers $p$ and $q$, with $0\leq p\leq q-2\leq t-2$, there exists $x\in C_{p}$ such that for all $y\in C_{q}$, $\{x,y\}$ is not an edge of $G$, then using the argument from part (a), we have $\ed[G]{C_{p},C_{q}}<\min\{|C_{p}|,|C_{q}|\}$. It implies that the number of edges is strictly less than $\M{C_{0},\ldots,C_{t}}$ and the result (b) follows.
\end{proof}

The above theorem informs us that the quest for a graph $G$ with lambda chromatic number $t$, which contains 
maximum number of edges, is boiled down to the search for a coloured partition $C_{0},\ldots,C_{t}$, originating 
from an optimal lambda colouring $c:V(G)\rightarrow\{0,\ldots,t\}$ where $\M{C_{0},\ldots,C_{t}}$ is maximum. Henceforth, our main focus is to search for such coloured partitions. 

\begin{proposition}\label{subgraph_relation}
Let $G'$ be a subgraph of the graph $G$. If the lambda chromatic numbers of $G'$ and $G$ are respectively $t'$ and $t$, then $t'\leq t$.
\end{proposition}
\begin{proof}
Let $\lambda:V(G)\rightarrow\mathbb{N}$ be a lambda colouring of $G$ and $\lambda':=\lambda_{/V(G')}$. Suppose 
$u,v\in V(G')$, then $\di[G']{u,v}\geq \di[G]{u,v}$. Consequently, 
$|\lambda'(u)-\lambda'(v)|+\di[G']{u,v}\geq|\lambda(u)-\lambda(v)|+\di[G]{u,v}$, which means $\lambda':V(G')\rightarrow\mathbb{N}$ is a lambda colouring of $G'$. Therefore
\begin{equation*}
t'=\min\left\{\underset{u\in V(G')}{\max}c(u): c\textup{ is a lambda colouring of } G'\right\}
\leq\underset{u\in V(G')}{\max}\lambda'(u)\leq\underset{u\in V(G)}{\max}\lambda(u)
\end{equation*}
The result follows, since the right hand side of the above inequality is true for any lambda colouring $\lambda$ of $G$.
\end{proof}

\begin{definition}
Let $G$ be a graph with lambda chromatic number $t$ and $C_{0},\ldots,C_{t}$ be the coloured partition of the vertex set $V(G)$ with respect to the lambda colouring $c$. The graph with vertex set $V(G)$ and edges of the form $\{u^{m}_{i},u^{p}_{i}\}$, where $1\leq i\leq\min\{|C_{m}|,|C_{p}|\}$, $C_{m}$ \& $C_{p}$ are non-empty and $0\leq m\leq p-2\leq t-2$, is called the \emph{edge standardised} graph of $G$ with respect to the lambda colour $c$ and denoted as $\mathscr{S}_{c}[G]$. Such edge standardised graph of $G$ contains $\M{C_{0},\ldots,C_{t}}$ number of edges.  
\end{definition}

Under the edge standardisation, the vertex set and its coloured partition remain invariant. An edge standardised graph $G$ with lambda chromatic number $t$ is a disjoint union of graphs of the form $\mathbb{T}_{t}$ (defined below). The following proposition ensures us that this technique does not reduce the number of edges.

\begin{proposition}\label{es_base}
Let $G$ be a graph with lambda chromatic number $t\geq3$ and $C_{0},\ldots,C_{t}$ be the coloured partition  of the vertex set of $G$ with respect to an optimal lambda colouring $c$.
Then such $c$ is also an optimal lambda colouring of $\mathscr{S}_{c}[G]$. Moreover, $|E(\mathscr{S}_{c}[G])|\geq|E(G)|$ and $\nabla_{c}(\mathscr{S}_{c}[G])=\nabla_{c}(G)$.  
\end{proposition}
\begin{proof}
As a part of the argument of the proof of Proposition~\ref{base}, we have showed that $\mathscr{S}_{c}[G]$ contains a subgraph $\mathbb{T}_{t}$ and the lambda chromatic number of $\mathbb{T}_{t}$ is $t$. Thus by using Proposition~\ref{subgraph_relation}, we conclude that the lambda chromatic number of $\mathscr{S}_{c}[G]$ is at least $t$. To get an upper bound, we associate  for each $u^{p}_{i}\mapsto p$, where $0\leq p\leq t$ and $1\leq i\leq |C_{p}|$. We call this map as 
$\Theta:V(\mathscr{S}_{c}[G])\rightarrow\mathbb{N}$ and arguing similarly as in a part of the proof of Proposition~\ref{base}, we conclude that $\Theta$ is a lambda colouring of the graph $\mathscr{S}_{c}[G]$. We note that 
$\underset{u\in V(\mathscr{S}_{c}[G])}{\max}\Theta(u)=t$. Hence the lambda chromatic number of such graph is at most $t$. It implies the lambda chromatic number of $\mathscr{S}_{c}[G]$ is $t$. It also implies that $\Theta$ is an optimal lambda colouring and for each $u\in V(G)$, $\Theta(u)=c(u)$. Hence $c$ is an optimal lambda colouring of $\mathscr{S}_{c}[G]$ and $\nabla_{c}(\mathscr{S}_{c}[G])=\nabla_{c}(G)$.

We note that for $0\leq p\leq q-2\leq t-2$, $\ed[G]{C_{p},C_{q}}\leq\min\{|C_{p}|,|C_{q}|\}$ and for $0\leq p\leq t-1$
$\ed[G]{C_{p},C_{p+1}}=0$. Since $\mathscr{S}_{c}[G]$ has exactly $\M{C_{0},\ldots,C_{t}}$ number of edges, therefore by Theorem~\ref{notation_M}, we have for $0\leq p\leq q-2\leq t-2$, 
$\ed[{\mathscr{S}_{c}[G]}]{C_{p},C_{q}}=\min\{|C_{p}|,|C_{q}|\}$ and for $0\leq p\leq t-1$, 
$\ed[{\mathscr{S}_{c}[G]}]{C_{p},C_{p+1}}=0$. Hence $|E(\mathscr{S}_{c}[G])|\geq|E(G)|$.
\end{proof}

\begin{definition}
Let $G$ be an edge standardised graph and with respect to the underlying lambda colouring $c$, let 
$C_{M}\in\mathfrak{M}_{c}(G)$ and $C_{m}\in\mathfrak{m}_{c}(G)$. The subgraph obtained by deleting the vertex 
$u^{M}_{|C_{M}|}$ and all the edges through that vertex of the graph $G$, is called as the \emph{edge deleted} graph. 
We denote such graph as $\mathscr{D}_{C_{M}}[G]$. Similarly, the graph obtained by adding the vertex 
$u^{m}_{|C_{m}|+1}$ and all possible edges of the form $\{u^{m}_{|C_{m}|+1}, u^{p}_{|C_{m}|+1}\}$, where 
\begin{align*}
\left\{\begin{array}{lcl}
        p\in\{2,\ldots,t\}&\textup{ if } & m=0\\
        p\in\{0,\ldots,t-2\}&\textup{ if } & m=t\\
        p\in\{0,\ldots,t\}\smallsetminus\{m-1,m,m+1\}&\textup{ if }& 1\leq m\leq t-1
       \end{array}\right.
\end{align*}
and $|C_{p}|\geq|C_{m}|+1$, with the graph $G$, is called as the \emph{edge inserted} graph. We denote such graph as $\mathscr{I}_{C_{m}}[G]$.
\end{definition}

\begin{proposition}\label{base}
Let $G$ be an edge standardised graph with lambda chromatic number $t\geq3$ and $C_{0},\ldots,C_{t}$ be the coloured partition  of the vertex set of $G$ with respect to the underlying optimal lambda colouring $c$. Then the lambda chromatic number of $\mathscr{I}_{C_{m}}[G]$ is $t$. If deletion of the vertex $u^{M}_{|C_{M}|}$ from the graph $G$ does not produce two consecutive holes in the graph $\mathscr{D}_{C_{M}}[G]$, then the lambda chromatic number of $\mathscr{D}_{C_{M}}[G]$ is $t$.
\end{proposition}
Note that by deletion of the vertex $u^{M}_{|C_{M}|}$ from the graph $G$, produces at least two consecutive holes in the graph $\mathscr{D}_{C_{M}}[G]$, we mean $|C_{M}|=1$ and either $C_{M+1}$ is empty or $C_{M-1}$ is empty.
\begin{proof}
We consider the subgraph induced by the set of vertices 
$\{u^{m}_{1}:0\leq m\leq t, C_{m}\neq\emptyset\}$ of the graph $G$. We refer such subgraph as $\mathbb{T}_{t}$. Since $C_{0}$ \& $C_{t}$ are non-empty sets and $t\geq3$, therefore $\{u^{0}_{1},u^{t}_{1}\}$ is an edge of $\mathbb{T}_{t}$. 

\noindent{\textsl{Claim} :} The lambda chromatic number of the graph $\mathbb{T}_{t}$ is $t$.
\begin{proof}[\tt{Proof of claim} :]\renewcommand{\qedsymbol}{}
We can directly verify the claim is true for $t=3$. So without loss of generality we assume $t\geq4$.
Clearly if $u^{p}_{1}$ and $u^{q}_{1}\in V(\mathbb{T}_{t})$, where $0\leq p\leq q-2\leq t-2$, then $\{u^{p}_{1},u^{q}_{1}\}$ is an edge of $\mathbb{T}_{t}$. Also if $u^{p}_{1}$ and $u^{p+1}_{1}\in V(\mathbb{T}_{t})$, where $0\leq p\leq t-1$, then there exists an $u^{q}_{1}\in V(\mathbb{T}_{t})$, where 
\begin{align*}
\left\{\begin{array}{lcl}
        q=0&\textup{ if } & 2\leq p\leq t-1\\
        q=t&\textup{ if } & 0\leq p\leq t-3,
       \end{array}\right.
\end{align*}
such that $\{u^{p}_{1},u^{q}_{1}\}$ and $\{u^{p+1}_{1},u^{q}_{1}\}$ are edges of $\mathbb{T}_{t}$. This implies for all 
$u,v\in V(\mathbb{T}_{t})$, $1\leq\di[\mathbb{T}_{t}]{u,v}\leq2$. Let $\lambda:V(\mathbb{T}_{t})\rightarrow\mathbb{N}$ be a lambda colouring. Then for all $u,v\in V(\mathbb{T}_{t})$, $\lambda(u)\neq\lambda(v)$. Note that $C_{m}=\emptyset$ for some $m$, with $1\leq m\leq t-1$, corresponds to a hole of the colouring $c$. Since $c$ is optimal lambda colouring, there does not exist two consecutive holes. Therefore if $C_{m}=\emptyset$ for some $m$, with $1\leq m\leq t-1$, then $\{u^{m-1}_{1},u^{m+1}_{1}\}$ is an edge and consequently $|\lambda(u^{m-1}_{1})-\lambda(u^{m+1}_{1})|\geq2$. This implies
\begin{equation*}
\max\left\{\lambda(u):u\in\{u^{t}_{1}\}\sqcup\N[u^{t}_{1}]\right\}\geq\max\{0,\ldots,t\}=t,
\end{equation*}
where $\N[u^{t}_{1}]:=\left\{u^{p}_{1}:0\leq p\leq t-2, \{u^{p}_{1},u^{t}_{1}\}\in E(\mathbb{T}_{t})\right\}$ denotes the non-empty set of neighbours of $u^{t}_{1}$. Hence for each lambda colouring $\mu:V(\mathbb{T}_{t})\rightarrow\mathbb{N}$ we have 
$\underset{u\in V(\mathbb{T}_{t})}{\max}\mu(u)\geq t$. Therefore the lambda chromatic number of $\mathbb{T}_{t}$ is at least $t$. Since $\mathbb{T}_{t}$ is a subgraph of $G$ and the lambda chromatic number of $G$ is $t$. Therefore by using Proposition~\ref{subgraph_relation} we have the lambda chromatic number of $\mathbb{T}_{t}$ is at most $t$. Hence the claim is established.
\end{proof}

Note that, without loss of generality, the graph $\mathscr{D}_{C_{M}}[G]$ contains $\mathbb{T}_{t}$ as its 
induced subgraph.Now deletion of the vertex $u^{M}_{|C_{M}|}$ from the graph $G$ does not produce two 
consecutive holes in the graph $\mathscr{D}_{C_{M}}[G]$. By using similar arguments as in the above claim and Proposition~\ref{subgraph_relation}, we have the lambda chromatic number of such graph is at least $t$. The graph $\mathscr{D}_{C_{M}}[G]$ is a subgraph of $G$ and the lambda chromatic number of $G$ is $t$. Hence again by using Proposition~\ref{subgraph_relation}, the lambda chromatic number of $\mathscr{D}_{C_{M}}[G]$ is at most $t$, which concludes the lambda chromatic number of $\mathscr{D}_{C_{M}}[G]$ is $t$. 

Since $G$ is a subgraph of the graph $\mathscr{I}_{C_{m}}[G]$, by using Proposition~\ref{subgraph_relation} we conclude that the lambda chromatic number of such graph is at least $t$. We associate $u^{m}_{|C_{m}|+1}\mapsto m$ and for each $u^{p}_{i}\mapsto p$, where $0\leq p\leq t$ and $1\leq i\leq|C_{p}|$. We call this map as $\Lambda:V(\mathscr{I}_{C_{m}}[G])\rightarrow\mathbb{N}$ and claim the following.

\noindent{\textsl{Claim} :} The map $\Lambda$ is a lambda colouring of $\mathscr{I}_{C_{m}}[G]$.

\begin{proof}[\tt{Proof of claim} :]\renewcommand{\qedsymbol}{}
Let $u^{p}_{i}$ and $u^{q}_{j}$ be two distinct vertices of $G':=\mathscr{I}_{C_{m}}[G]$, where $0\leq p\leq t$, $0\leq q\leq t$, $i\geq 1$ and $j\geq1$. We note that if $i\neq j$, then $\di[G']{u^{p}_{i},u^{q}_{j}}\neq 0,1,2$. Also if $i=j$, then
\begin{align*}
\di[G']{u^{p}_{i},u^{q}_{i}}=\left\{\begin{array}{lcl}
              \left\{\begin{array}{lcl}
                2& \textnormal{if} & p=0, q=1 \\
                1& \textnormal{if} & p=0, 2\leq q\leq t
            \end{array}\right.\\
            \left\{\begin{array}{lcl}
                2& \textnormal{if} & p=t, q=t-1 \\
                1& \textnormal{if} & p=t, 0\leq q\leq t-2
            \end{array}\right.\\
            \left\{\begin{array}{lcl}
		3& \textnormal{if} & p=1,q=2 \textnormal{ and } t=3\\
                2& \textnormal{if} & 1\leq p\leq t-1, q=p-1 \textnormal{ or } p+1\textnormal{ and } t\geq4\\
                1& \textnormal{if} & 1\leq p\leq t-1, 0\leq q\leq p-2 \textnormal{ or } p+2\leq q\leq t
            \end{array}\right.\\
                       \end{array}\right.
\end{align*}
Hence for two distinct vertices $u$ and $v$ of $G'$, we have $|\Lambda(u)-\Lambda(v)|+\di[G']{u,v}\geq3$, which establishes the claim. 
\end{proof}
\noindent We note that $\underset{u\in V(\mathscr{I}_{C_{m}}[G])}{\max}\Lambda(u)=t$. Hence from the above claim the lambda chromatic number of such graph is at most $t$, which concludes the lambda chromatic number of $\mathscr{I}_{C_{m}}[G]$ is $t$. 
\end{proof}

An alternative proof of the first claim of Proposition~\ref{base} follows easily by using Theorem~1.1 from \cite{MR1310873}. A \emph{path} of length $k$, where $k$ is a positive integer, in a graph $G$ is a sequence $\{u_{i}\}_{i=1}^{k+1}$ of distinct vertices such that for $1\leq i\leq k$, $\{u_{i},u_{i+1}\}$ is an edge of $G$. 
The vertices $u_{1}$ and $u_{k+1}$ is called the \emph{initial} and \emph{terminal} vertices, respectively. A \emph{path covering} of $G$, denoted as $\mathscr{C}(G)$, is a collection of vertex disjoint paths in $G$ such that for each vertex $u\in V(G)$ there exists a (unique) $C\in\mathscr{C}(G)$ such that $u\in C$. A \emph{minimum path covering} of $G$ is a path covering of $G$ with minimum cardinality and the \emph{path covering number} 
$\tau_{p}(G)$ of $G$ is the cardinality of a minimum path covering of $G$. The Theorem~1.1 of \cite{MR1310873}, states that the path covering number of the complement graph $\overline{G}$ of an $n-$vertex graph $G$ is 
$\tau_{p}(\overline{G})$. Then one of the following holds.
\begin{itemize}
\item $\tau_{p}(\overline{G})=1$ if and only if the lambda chromatic number of $G$ is less or equals to $n-1$.
\item $\tau_{p}(\overline{G})\geq2$ if and only if the lambda chromatic number of $G$ is $n+\tau_{p}(\overline{G})-2$.
\end{itemize}

\begin{proof}[\tt{Alternative proof of the first claim of Proposition~\ref{base}} :]\renewcommand{\qedsymbol}{}
If $|V(\mathbb{T}_{t})|=t+1$, then for each integer $m$, with $0\leq m\leq t$, $C_{m}$ is a non-empty set. Hence the graph $\mathbb{T}_{t}$ is isomorphic to the graph $\mathbb{G}_{t}$. From Theorem~\ref{lambda_G_n}, we conclude that the lambda chromatic number of $\mathbb{T}_{t}$ is $t$. Hence we are done this case. Now if $|V(\mathbb{T}_{t})|=t+1-r$ for some positive integer $r$, then among the $t+1$ coloured classes $C_{0},\ldots,C_{t}$ exactly $r$ coloured classes are empty. We note that $C_{0}$ and $C_{t}$ can never be empty. Since $c$ is optimal lambda colouring of $G$, we have if $C_{m}$ is empty for some integer $m$, with $1\leq m\leq t-1$, then both $C_{m-1}$ and $C_{m+1}$ are non-empty. These imply the complement graph $\overline{\mathbb{T}_{t}}$ of $\mathbb{T}_{t}$ is an union of $r+1$ vertex disjoint paths (path graphs) $P_{0},\ldots,P_{r}$. 
We note that such paths $P_{0},\ldots,P_{t}$ form a path covering of $\overline{\mathbb{T}_{t}}$. Hence $\tau_{p}(\overline{\mathbb{T}_{t}})\leq r+1$. Again we note that for any two paths $P$ and $Q$ from $P_{0},\ldots,P_{r}$,  there does not exist any edge of the form $\{u,v\}$, where $u,v$ are vertices of the paths $P$ and $Q$ respectively. With such property, it implies that $P_{0},\ldots,P_{r}$ is the only path covering of $\overline{\mathbb{T}_{t}}$ with cardinality less or equals to $r+1$. Hence $\tau_{p}(\overline{\mathbb{T}_{t}})\geq r+1$. Consequently, $\tau_{p}(\overline{\mathbb{T}_{t}})=r+1$. Now by using Theorem~1.1 of \cite{MR1310873}, we have the lambda chromatic number of $\mathbb{T}_{t}$ is $t+1-r+(r+1)-2=t$.
\end{proof}

\begin{remark}
It is obvious that an optimal lambda colouring $c:V(G)\rightarrow\{0,\ldots,t\}$ induces the coloured partition 
$C_{0},\ldots,C_{t}$ of the vertex set of a graph $G$. In addition, if $G$ is an edge standardised graph, than both the graphs $\mathscr{D}_{C_{M}}[G]$ and $\mathscr{I}_{C_{m}}[G]$, where $C_{M}\in\mathfrak{M}_{c}(G)$ and $C_{m}\in\mathfrak{m}_{c}(G)$, are edge standardised graphs. Also both have lambda chromatic number $t$. Therefore it is completely legitimate to study about the edge standardised graph $\mathscr{I}_{C_{m}}\mathscr{D}_{C_{M}}[G]$ ($:=\mathscr{I}_{C_{m}}[\mathscr{D}_{C_{M}}[G]]$). Such graph is obtained after deletion of the vertex $u^{M}_{|C_{M}|}$ along with all the edges through that vertex of the edge standardised graph $G$ and then inserting the vertex $u^{m}_{|C_{m}|+1}$ as well as all possible edges of the form $\{u^{m}_{|C_{m}|+1}, u^{p}_{|C_{m}|+1}\}$, where 
\begin{align*}
\left\{\begin{array}{lcl}
        p\in\{2,\ldots,t\}&\textup{ if } & m=0\\
        p\in\{0,\ldots,t-2\}&\textup{ if } & m=t\\
        p\in\{0,\ldots,t\}\smallsetminus\{m-1,m,m+1\}&\textup{ if }& 1\leq m\leq t-1
       \end{array}\right.
\end{align*}
and $|C_{p}|\geq|C_{m}|+1$.

We start with the coloured partition $C_{0},\ldots,C_{m},\ldots,C_{M},\ldots,C_{t}$ of $G$. Such coloured partition 
has changed to $C_{0},\ldots,C_{m},\ldots,C_{M}\smallsetminus\{u^{M}_{|C_{M}|}\},\ldots,C_{t}$
in the graph $\mathscr{D}_{C_{M}}[G]$ and such coloured partition has changed to
$C_{0},\ldots,C_{m}\sqcup\{u^{m}_{|C_{m}|+1}\},\ldots,C_{M},\ldots,C_{t}$ in the graph $\mathscr{I}_{C_{m}}[G]$. Therefore using Proposition~\ref{base} repeatedly, we say that during the edge deletion and insertion procedure, the coloured partition $C_{0},\ldots,C_{m},\ldots,C_{M},\ldots,C_{t}$ of $G$ has changed in the graph $\mathscr{I}_{C_{m}}\mathscr{D}_{C_{M}}[G]$, which is 
\begin{equation*}
C_{0},\ldots,C_{m}\sqcup\{u^{m}_{|C_{m}|+1}\},\ldots,C_{M}\smallsetminus\{u^{M}_{|C_{M}|}\},\ldots,C_{t},
\end{equation*}
as long as deleting $u^{M}_{|C_{M}|}$ does not produce two consecutive empty classes (i.e. holes) in $\mathscr{D}_{C_{M}}[G]$. For the sake of simplicity we refer each of the lambda colouring of 
$\mathscr{D}_{C_{M}}[G]$, $\mathscr{I}_{C_{m}}[G]$ and $\mathscr{I}_{C_{m}}\mathscr{D}_{C_{M}}[G]$ inducing the aforementioned vertex partition as $c$. Though the the lambda chromatic number remains invariant during the edge deletion and insertion procedure, but $\mathfrak{M}_{c}(G)$ and $\mathfrak{m}_{c}(G)$ change to 
$\mathfrak{M}_{c}(\mathscr{D}_{C_{M}}[G])$ and $\mathfrak{m}_{c}(\mathscr{I}_{C_{m}}[G])$ respectively.
\end{remark}

\begin{definition}
Let $G$ is a graph with lambda chromatic number $t$ and an optimal lambda colouring $c:V(G)\rightarrow\{0,\ldots,t\}$ induce the coloured partition $C_{0},\ldots,C_{t}$ of the vertex set of $G$. Then 
\begin{align*}
\mathsf{P}_{G}(C_{m})=\left\{\begin{array}{lcl}
                         \{C_{1}\} & \textnormal{if} & m=0\\
                         \{C_{t-1}\} & \textnormal{if} & m=t\\ 
                         \{C_{m-1},C_{m+1}\}& \textnormal{if} & 1\leq m\leq t-1
                         \end{array}\right.
\end{align*}
denotes the \emph{prohibited zone} of the colour class $C_{m}$ in the graph $G$. We use the term 
``prohibited'' is due to the property that for each $x\in C\in\mathsf{P}_{G}(C_{m})$ there does not 
exist $y\in C_{m}$ such that $\{x,y\}$ is an edge of $G$.
\end{definition}

The solution of Question~(b) relies on the following three lemmas. 

\begin{lemma}\label{maxformula}
Let $G$ be an edge standardised graph with lambda chromatic number $t\geq3$ and $C_{0},\ldots,C_{t}$ be the coloured partition  of the vertex set of $G$ with respect to the underlying optimal lambda colouring $c$ and $\nabla_{c}(G)\geq2$. Then for $C_{M}\in\mathfrak{M}_{c}(G)$ the lambda chromatic number of $\mathscr{D}_{C_{M}}[G]$ is $t$ and the following formula holds: 
\begin{equation*}
|E(G)|=|E(\mathscr{D}_{C_{M}}[G])|+|\mathfrak{M}_{c}(G)|-1-|\mathfrak{M}_{c}(G)\cap\mathsf{P}_{G}(C_{M})|.
\end{equation*}
\end{lemma}            
\begin{proof}
We note that $\nabla_{c}(G)\geq2$ ensures that the deletion of the vertex $u^{M}_{|C_{M}|}$ from the graph $G$ does not produce two consecutive holes in the graph $\mathscr{D}_{C_{M}}[G]$. Therefore using the arguments in Proposition~\ref{base}, we conclude that the lambda chromatic number of $\mathscr{D}_{C_{M}}[G]$ is $t$. Since $G$ is an edge standardised graph, 
we have the following: 
\begin{align*}\label{deletion}
|E(G)|-|E(\mathscr{D}_{C_{M}}[G])|=\left\{\begin{array}{lcl}
                         |\mathfrak{M}_{c}(G)|-1 & \textnormal{if} & |\mathfrak{M}_{c}(G)\cap\mathsf{P}_{G}(C_{M})|=0\\
                         |\mathfrak{M}_{c}(G)|-2 & \textnormal{if} & |\mathfrak{M}_{c}(G)\cap\mathsf{P}_{G}(C_{M})|=1\\ 
                         |\mathfrak{M}_{c}(G)|-3 & \textnormal{if} & |\mathfrak{M}_{c}(G)\cap\mathsf{P}_{G}(C_{M})|=2 
                         \end{array}\right.
\end{align*}
Hence the formula holds.
\end{proof}

\begin{lemma}\label{minformula}
Let $G$ be an edge standardised graph with lambda chromatic number $t\geq3$ and $C_{0},\ldots,C_{t}$ be the coloured partition  of the vertex set of $G$ with respect to the underlying optimal lambda colouring $c$. Then for $C_{m}\in\mathfrak{m}_{c}(G)$, the following formula holds: 
\begin{equation*}
|E(\mathscr{I}_{C_{m}}[G])|=|E(G)|+t+1-|\mathfrak{m}_{c}(G)\cup\mathsf{P}_{G}(C_{m})|.
\end{equation*}
\end{lemma}
\begin{proof}
The result holds since there are exactly $t+1-|\mathfrak{m}_{c}(G)\cup\mathsf{P}_{G}(C_{m})|$ edges which are edges of $\mathscr{I}_{C_{m}}[G]$ but not edges of $G$. The counting is as follows. Since $G$ is an edge standardised graph, the vertex $u^{m}_{|C_{m}|+1}$ is adjacent with the vertex of the form $u^{p}_{|C_{m}|+1}$, where $0\leq p\leq t$ and 
$|C_{p}|\geq|C_{m}|+1$. Hence there are at most $t+1$ such vertices. But if 
$C_{q}\in\mathfrak{m}_{c}(G)\cup\mathsf{P}_{G}(C_{m})$, then there is no such vertex of the form 
$u^{q}_{|C_{m}|+1}\in C_{q}$ such that $\{u^{m}_{|C_{m}|+1},u^{q}_{|C_{m}|+1}\}$ is an edge of $\mathscr{I}_{C_{m}}[G]$ and vice versa. 
\end{proof}

The development of this section and the following one is all about searching an $n-$vertex graph with lambda chromatic number $t\geq3$ containing maximum number of edges. In this regard, we now focus on the following question.  

\begin{question}
Let $G$ be an $n-$vertex graph $G$ with lambda chromatic number $t\geq3$. If $\nabla_{c}(G)\geq2$ for some optimal lambda colouring $c$ of $G$, is it possible that $G$ contains maximum number of edges?
\end{question}

We answer the above question affirmatively. More precisely, we answer that for $t=3$ and $t=4$ we can find such families of graphs. However, we can not find any such graph when $t\geq5$. Formally we explain it in the following results.  

\begin{lemma}\label{atmostone}
Let $G$ be an edge standardised graph with lambda chromatic number $t\geq3$ and $C_{0},\ldots,C_{t}$ be the coloured partition  of the vertex set of $G$ with respect to the underlying optimal lambda colouring $c$. If $G$ contains maximum number of edges, then $\nabla_{c}(G)\leq1$ or all but at most one $C_{i}$ are members of 
$\mathfrak{M}_{c}(G)\sqcup\mathfrak{m}_{c}(G)$ (where $0\leq i\leq t$). 
\end{lemma}
\begin{proof}
Suppose it does not hold that $\nabla_{c}(G)\leq1$ or all but at most one $C_{i}$ are members of $\mathfrak{M}_{c}(G)\sqcup\mathfrak{m}_{c}(G)$. This implies $\nabla_{c}(G)\geq2$ and $2\leq|\mathfrak{M}_{c}(G)|+|\mathfrak{m}_{c}(G)|\leq t-1$. We choose $A\in\mathfrak{M}_{c}(G)$ and $B\in\mathfrak{m}_{c}(G)$ and construct two edge standardised graphs viz. $\mathscr{D}_{A}[G]$ and $\mathscr{I}_{B}\mathscr{D}_{A}[G]$. 
We note that $\nabla_{c}(G)\geq2$ implies $\mathfrak{m}_{c}(\mathscr{D}_{A}[G])=\mathfrak{m}_{c}(G)$. Hence $B\in\mathfrak{m}_{c}(\mathscr{D}_{A}[G])$. Also $|\mathsf{P}_{\mathscr{D}_{A}[G]}(B)|=|\mathsf{P}_{G}(B)|$. Now by using Lemma~\ref{maxformula}, Lemma~\ref{minformula} and the inequality 
\begin{equation*}
3\leq|\mathfrak{M}_{c}(G)|+|\mathfrak{m}_{c}(G)|+|\mathsf{P}_{G}(B)|=|\mathfrak{M}_{c}(G)|+
|\mathfrak{m}_{c}(\mathscr{D}_{A}[G])|+|\mathsf{P}_{\mathscr{D}_{A}[G]}(B)|\leq t+1
\end{equation*}
we get, 
\begin{align*}
|E(\mathscr{I}_{B}\mathscr{D}_{A}[G])|-|E(G)|&=t+2+|\mathfrak{M}_{c}(G)\cap\mathsf{P}_{G}(A)|+|\mathfrak{m}_{c}(\mathscr{D}_{A}[G])\cap\mathsf{P}_{\mathscr{D}_{A}[G]}(B)|\\
&\hspace{2cm}-\left(|\mathfrak{M}_{c}(G)|+|\mathfrak{m}_{c}(\mathscr{D}_{A}[G])|+|\mathsf{P}_{\mathscr{D}_{A}[G]}(B)|\right)\tag{$\star$}\label{star}\\
&\geq1+|\mathfrak{M}_{c}(G)\cap\mathsf{P}_{G}(A)|+|\mathfrak{m}_{c}(\mathscr{D}_{A}[G])\cap\mathsf{P}_{\mathscr{D}_{A}[G]}(B)|.
\end{align*}
Both the graphs $\mathscr{I}_{B}\mathscr{D}_{A}[G]$ and $G$ are edge standardised graphs and 
$|E(\mathscr{I}_{B}\mathscr{D}_{A}[G])|\geq|E(G)|+1$. Hence $G$ does not contain maximum number of edges, a contradiction arises and the result follows.
\end{proof}

During the transformation of the edge standardised graph $G$ into $\mathscr{I}_{B}\mathscr{D}_{A}[G]$, the Lemma~\ref{maxformula}, Lemma~\ref{minformula}, Lemma~\ref{atmostone} and the the equation \eqref{star} provide us the opportunity to understand the role of the intermediate graph $\mathscr{D}_{A}[G]$. It keeps track the change of the edge distribution from a quantitative point of view. 

\begin{proposition}\label{nablageq2}
Let $G$ be an edge standardised graph with lambda chromatic number $t\geq3$ and $C_{0},\ldots,C_{t}$ be the coloured partition  of the vertex set of $G$ with respect to the underlying optimal lambda colouring $c$. If $\nabla_{c}(G)\geq2$ and $G$ contains maximum number of edges, then exactly one of the following holds.
\begin{enumerate}[\normalfont(a)]
\item For all $A\in\mathfrak{M}_{c}(G)$ and $B\in\mathfrak{m}_{c}(G)$, 
$|\mathfrak{M}_{c}(G)\cap\mathsf{P}_{G}(A)|=0$, $|\mathfrak{m}_{c}(G)\cap\mathsf{P}_{G}(B)|=0$,\\
whenever $|\mathfrak{m}_{c}(G)|+|\mathfrak{M}_{c}(G)|=t$.  
\item  For all $A\in\mathfrak{M}_{c}(G)$ and $B\in\mathfrak{m}_{c}(G)$,
\begin{equation}\tag{$\star\star$}\label{starstar}
1+|\mathfrak{M}_{c}(G)\cap\mathsf{P}_{G}(A)|+|\mathfrak{m}_{c}(G)\cap\mathsf{P}_{G}(B)|\leq|\mathsf{P}_{G}(B)|,
\end{equation}
whenever $|\mathfrak{m}_{c}(G)|+|\mathfrak{M}_{c}(G)|=t+1$.
\end{enumerate}
\end{proposition}
\begin{proof}
Let $|\mathfrak{m}_{c}(G)|+|\mathfrak{M}_{c}(G)|=t$. Suppose for some $A\in\mathfrak{M}_{c}(G)$ and $B\in\mathfrak{m}_{c}(G)$ we have 
\begin{equation*}
|\mathfrak{M}_{c}(G)\cap\mathsf{P}_{G}(A)|+|\mathfrak{m}_{c}(G)\cap\mathsf{P}_{G}(B)|\geq1.
\end{equation*}
Then we construct two edge standardised graphs viz. $\mathscr{D}_{A}[G]$ and $\mathscr{I}_{B}\mathscr{D}_{A}[G]$. 
We note that $\nabla_{c}(G)\geq2$ implies $\mathfrak{m}_{c}(\mathscr{D}_{A}[G])=\mathfrak{m}_{c}(G)$. Hence $B\in\mathfrak{m}_{c}(\mathscr{D}_{A}[G])$ and 
$|\mathfrak{m}_{c}(\mathscr{D}_{A}[G])\cap\mathsf{P}_{\mathscr{D}_{A}[G]}(B)|=|\mathfrak{m}_{c}(G)\cap\mathsf{P}_{G}(B)|$. Now by using Lemma~\ref{maxformula}, Lemma~\ref{minformula}, \eqref{star} of Lemma~\ref{atmostone} and the inequality 
\begin{equation*}
|\mathfrak{M}_{c}(G)|+|\mathfrak{m}_{c}(G)|+|\mathsf{P}_{G}(B)|=|\mathfrak{M}_{c}(G)|+
|\mathfrak{m}_{c}(\mathscr{D}_{A}[G])|+|\mathsf{P}_{\mathscr{D}_{A}[G]}(B)|\leq t+2
\end{equation*}
we get, 
\begin{align*}
|E(\mathscr{I}_{B}\mathscr{D}_{A}[G])|-|E(G)|&\geq|\mathfrak{M}_{c}(G)\cap\mathsf{P}_{G}(A)|+|\mathfrak{m}_{c}(\mathscr{D}_{A}[G])\cap\mathsf{P}_{\mathscr{D}_{A}[G]}(B)|\\
&=|\mathfrak{M}_{c}(G)\cap\mathsf{P}_{G}(A)|+|\mathfrak{m}_{c}(G)\cap\mathsf{P}_{G}(B)|\geq1.
\end{align*}
Both the graphs $\mathscr{I}_{B}\mathscr{D}_{A}[G]$ and $G$ are edge standardised graphs and 
$|E(\mathscr{I}_{B}\mathscr{D}_{A}[G])|\geq|E(G)|+1$. Hence $G$ does not contain maximum number of edges, a contradiction arises. This concludes (a).

Let $|\mathfrak{m}_{c}(G)|+|\mathfrak{M}_{c}(G)|=t+1$. Suppose for some $A\in\mathfrak{M}_{c}(G)$ and 
$B\in\mathfrak{m}_{c}(G)$ we have 
\begin{equation*}
|\mathfrak{M}_{c}(G)\cap\mathsf{P}_{G}(A)|+|\mathfrak{m}_{c}(G)\cap\mathsf{P}_{G}(B)|\geq|\mathsf{P}_{G}(B)|.
\end{equation*}
Then we construct two edge standardised graphs viz. $\mathscr{D}_{A}[G]$ and $\mathscr{I}_{B}\mathscr{D}_{A}[G]$. 
We note that $\nabla_{c}(G)\geq2$ implies $\mathfrak{m}_{c}(\mathscr{D}_{A}[G])=\mathfrak{m}_{c}(G)$. Hence $B\in\mathfrak{m}_{c}(\mathscr{D}_{A}[G])$ and 
$|\mathfrak{m}_{c}(\mathscr{D}_{A}[G])\cap\mathsf{P}_{\mathscr{D}_{A}[G]}(B)|=|\mathfrak{m}_{c}(G)\cap\mathsf{P}_{G}(B)|$. 
Also $|\mathsf{P}_{\mathscr{D}_{A}[G]}(B)|=|\mathsf{P}_{G}(B)|$. Now by using Lemma~\ref{maxformula}, Lemma~\ref{minformula} and \eqref{star} of Lemma~\ref{atmostone} we get, 
\begin{align*}
|E(\mathscr{I}_{B}\mathscr{D}_{A}[G])|-|E(G)|&=1+|\mathfrak{M}_{c}(G)\cap\mathsf{P}_{G}(A)|+
|\mathfrak{m}_{c}(\mathscr{D}_{A}[G])\cap\mathsf{P}_{\mathscr{D}_{A}[G]}(B)|-|\mathsf{P}_{\mathscr{D}_{A}[G]}(B)|\\
&=1+|\mathfrak{M}_{c}(G)\cap\mathsf{P}_{G}(A)|+|\mathfrak{m}_{c}(G)\cap\mathsf{P}_{G}(B)|-|\mathsf{P}_{G}(B)|\geq1.
\end{align*}
Both the graphs $\mathscr{I}_{B}\mathscr{D}_{A}[G]$ and $G$ are edge standardised graphs and 
$|E(\mathscr{I}_{B}\mathscr{D}_{A}[G])|\geq|E(G)|+1$. Hence $G$ does not contain maximum number of edges, a contradiction arises. This concludes (b).
\end{proof}

\begin{remark}
The \eqref{starstar} condition of Proposition~\ref{nablageq2} expounds the followings.
\begin{itemize}
\item Any three consecutive members (i.e. members of the form $C_{i},C_{i+1},C_{i+2}$, where $0\leq i\leq t-2$) of 
$C_{0},\ldots,C_{t}$ must contain at least one member of $\mathfrak{M}_{c}(G)$ and at least one member of 
$\mathfrak{m}_{c}(G)$.
\item Occurrence of two consecutive members (i.e. members of the form $C_{i},C_{i+1}$, where $0\leq i\leq t-1$) in $\mathfrak{M}_{c}(G)$ forbids occurrence of two consecutive members in $\mathfrak{m}_{c}(G)$ and vice-versa. 
\end{itemize}
\end{remark}

Suppose $G$ is an edge standardised graph with lambda chromatic number $t\geq3$ and $C_{0},\ldots,C_{t}$ be the coloured partition  of the vertex set of $G$ with respect to the underlying optimal lambda colouring 
$c:V(G)\rightarrow\{0,\ldots,t\}$. Also let $\nabla_{c}(G)\geq2$. From \eqref{star} of Lemma~\ref{atmostone}, if $G$ has maximum number of edges and $|\mathfrak{M}_{c}(G)|+|\mathfrak{m}_{c}(G)|=t$, then $|E(\mathscr{I}_{B}\mathscr{D}_{A}[G])|=|E(G)|$ for all $A\in\mathfrak{M}_{c}(G)$ and $B\in\mathfrak{m}_{c}(G)$. But if $G$ has maximum number of edges and 
$|\mathfrak{M}_{c}(G)|+|\mathfrak{m}_{c}(G)|=t+1$, then either 
$|E(\mathscr{I}_{B}\mathscr{D}_{A}[G])|=|E(G)|$ or $|E(\mathscr{I}_{B}\mathscr{D}_{A}[G])|=|E(G)|-1$, where $A\in\mathfrak{M}_{c}(G)$ and $B\in\mathfrak{m}_{c}(G)$.

Here we derive the necessary conditions to prohibit the increase in the number of edges. The first two results (Proposition~\ref{sum_t} and Proposition~\ref{sum_t+1_gain-0}) are about the transformations that keep the number of edges invariant. However, the third one (Proposition~\ref{sum_t+1_gain-1}) is about the transformations when the 
number of edges reduces. Eventually, these lead to stationary conditions (vide the definition of stationary graph) in the subsequent development.

\begin{proposition}\label{sum_t}
Let $G$ be an edge standardised graph with lambda chromatic number $t\geq3$ and $C_{0},\ldots,C_{t}$ be the coloured partition  of the vertex set of $G$ with respect to the underlying optimal lambda colouring $c$. If 
$\nabla_{c}(G)\geq2$ and $G$ contains maximum number of edges with $|\mathfrak{m}_{c}(G)|+|\mathfrak{M}_{c}(G)|=t$, then 
$t=3$ or $t=4$; moreover $|\mathfrak{m}_{c}(G)|=1$ (and consequently), $\nabla_{c}(G)=2$ and $\mathfrak{m}_{c}(G)=\{C_{i}\}$, where $1\leq i\leq t-1$.
\end{proposition}
\begin{proof}
Since $\nabla_{c}(G)>0$, there exists $A\in\mathfrak{M}_{c}(G)$. Let for some $X\in\mathfrak{m}_{c}(G)$ 
$|\mathsf{P}_{G}(X)|=1$. Now both the graphs $\mathscr{I}_{X}\mathscr{D}_{A}[G]$ and $G$ are edge standardised 
graphs. By \eqref{star} of Lemma~\ref{atmostone}, we have $|E(\mathscr{I}_{X}\mathscr{D}_{A}[G])|\geq|E(G)|+1$. Hence $G$ does not contain maximum number of edges. A contradiction arises. Hence for each $B\in\mathfrak{m}_{c}(G)$, $|\mathsf{P}_{G}(B)|=2$. Consequently, $C_{i}\in\mathfrak{m}_{c}(G)$ implies $1\leq i\leq t-1$. 

We now proof the following claim.

\noindent{\textsl{Claim} :} With these conditions, $|\mathfrak{M}_{c}(G)|\geq2$. 
\begin{proof}[\tt{Proof of claim} :]\renewcommand{\qedsymbol}{}
Suppose $|\mathfrak{M}_{c}(G)|\leq1$, then $|\mathfrak{M}_{c}(G)|=1$. If $Y$ be the unique member of 
$\mathfrak{M}_{c}(G)$, then $|E(\mathscr{I}_{X}\mathscr{D}_{Y}[G])|\geq|E(G)|+1$ for each $X\in\mathfrak{m}_{c}(G)$. Since  $G$ and $\mathscr{I}_{X}\mathscr{D}_{Y}[G]$ both are edge standardised graphs, we have $G$ does not contain maximum number of edges, a contradiction arises. Hence the claim is established.
\end{proof}

Suppose with these conditions, we have $|\mathfrak{m}_{c}(G)|\geq2$. Using the above claim we choose $A$, 
$A'\in\mathfrak{M}_{c}(G)$ and $B$, $B'\in\mathfrak{m}_{c}(G)$. We construct the graphs $\mathscr{D}_{A}[G]$, $\mathscr{D}_{A'}\mathscr{D}_{A}[G]$, $\mathscr{I}_{B}\mathscr{D}_{A'}\mathscr{D}_{A}[G]$ and $\mathscr{I}_{B'}\mathscr{I}_{B}\mathscr{D}_{A'}\mathscr{D}_{A}[G]$. Now $\nabla_{c}(G)\geq2$. Also 
$|\mathfrak{M}_{c}(G)|\geq2$, therefore $\nabla_{c}(\mathscr{D}_{A}[G])\geq2$ and by Lemma~\ref{maxformula}, lambda chromatic number of $\mathscr{D}_{A'}\mathscr{D}_{A}[G]$ is $t$. Then by using Proposition~\ref{base}, $\mathscr{I}_{B}\mathscr{D}_{A'}\mathscr{D}_{A}[G]$ and $\mathscr{I}_{B'}\mathscr{I}_{B}\mathscr{D}_{A'}\mathscr{D}_{A}[G]$ are of lambda chromatic number $t$. 
We note the following :
\begin{align*}
|\mathfrak{M}_{c}(\mathscr{D}_{A}[G])|&=|\mathfrak{M}_{c}(G)|-1,\\
|\mathfrak{m}_{c}(\mathscr{D}_{A'}\mathscr{D}_{A}[G])|&=|\mathfrak{m}_{c}(G)|\textup{ and}\\ 
|\mathfrak{m}_{c}(\mathscr{I}_{B}\mathscr{D}_{A'}\mathscr{D}_{A}[G])|&=|\mathfrak{m}_{c}(G)|-1.
\end{align*}
Using these, Lemma~\ref{maxformula} and Lemma~\ref{minformula} we have, 
\begin{align*}
|E(\mathscr{I}_{B'}\mathscr{I}_{B}\mathscr{D}_{A'}\mathscr{D}_{A}[G])|-|E(G)|&=2t+6+|\mathfrak{M}_{c}(\mathscr{D}_{A}[G])\cap\mathsf{P}_{\mathscr{D}_{A}[G]}(A')|+|\mathfrak{M}_{c}(G)\cap\mathsf{P}_{G}(A)|+\\
&\hspace{1cm}+|\mathfrak{m}_{c}(\mathscr{D}_{A'}\mathscr{D}_{A}[G])\cap\mathsf{P}_{\mathscr{D}_{A'}\mathscr{D}_{A}[G]}(B)|+\\
&\hspace{1cm}+|\mathfrak{m}_{c}(\mathscr{I}_{B}\mathscr{D}_{A'}\mathscr{D}_{A}[G])\cap\mathsf{P}_{\mathscr{I}_{B}\mathscr{D}_{A'}\mathscr{D}_{A}[G]}(B')|-\\
&\hspace{1cm}-2(|\mathfrak{M}_{c}(G)|+|\mathfrak{m}_{c}(G)|)\\
&\hspace{1cm}-|\mathsf{P}_{\mathscr{D}_{A'}\mathscr{D}_{A}}(B)|-
|\mathsf{P}_{\mathscr{I}_{B}\mathscr{D}_{A'}\mathscr{D}_{A}}(B')|\\
&\geq2
\end{align*}
Both the graphs $\mathscr{I}_{B'}\mathscr{I}_{B}\mathscr{D}_{A'}\mathscr{D}_{A}[G]$ and $G$ are edge standardised graphs. Also $|E(\mathscr{I}_{B'}\mathscr{I}_{B}\mathscr{D}_{A'}\mathscr{D}_{A}[G])|\geq|E(G)|+2$. Hence $G$ does not contain maximum number of edges. A contradiction arises. This proves $|\mathfrak{m}_{c}(G)|=1$.

Let $B$ be the unique member of $\mathfrak{m}_{c}(G)$. Using the aforementioned arguments we have 
$|\mathsf{P}_{G}(B)|=2$. Also we have 
$|\mathfrak{M}_{c}(G)|+|\mathfrak{m}_{c}(G)|=t$ and hence by using (a) of Proposition~\ref{nablageq2}, we have for all $A\in\mathfrak{M}_{c}(G)$, $|\mathfrak{M}_{c}(G)\cap\mathsf{P}_{G}(A)|=0$ and 
$|\mathfrak{m}_{c}(G)\cap\mathsf{P}_{G}(B)|=0$. These together imply 
$2\leq|\mathfrak{M}_{c}(G)|\leq3$. Hence $4\leq t+1\leq5$, i.e. $3\leq t\leq4$.

Let $A\in\mathfrak{M}_{c}(G)$ and $C$ denote the unique coloured class with $|B|+1\leq|C|\leq|A|-1$. We consider the edge standardised graph $G':=\mathscr{I}_{B}\mathscr{D}_{A}[G]$. Suppose $\nabla_{c}(G)\geq3$, then we conclude 
$|C|-|B|\geq2$ or $|A|-|C|\geq2$.

\noindent{\textsf{Case} I : } Suppose $|C|-|B|\geq2$.

\noindent For this case, we have $|\mathfrak{M}_{c}(G)|=t-1\geq2$. Consequently $|\mathfrak{m}_{c}(G')|=1$ and 
$|\mathfrak{M}_{c}(G')|+|\mathfrak{m}_{c}(G')|=t-1$. Hence we conclude by Lemma~\ref{atmostone}, $G'$ can not contain maximum number of edges and by \eqref{star}, $|E(G')|\geq|E(G)|$. But $G$ contains maximum number of edges, which leads to a contradiction. 

\noindent{\textsf{Case} II : } Suppose $|A|-|C|\geq2$.

\noindent If $|C|-|B|\geq2$, then $|\mathfrak{M}_{c}(G)|=t-1\geq2$ and consequently 
$|\mathfrak{M}_{c}(G')|+|\mathfrak{m}_{c}(G')|=t-1$. With a similar argument as in Case~I, we have a contradiction. If 
$|C|-|B|=1$, then $|\mathfrak{M}_{c}(G')|+|\mathfrak{m}_{c}(G')|=t$ but $|\mathfrak{m}_{c}(G')|=2$. Hence by \eqref{star} of Lemma~\ref{atmostone}, we conclude $|E(G')|\geq|E(G)|$. Since $G'$ is an edge standardised graph and $G$ contains maximum number of edges, we conclude that $G'$ contains maximum number of edges. A contradiction arises, as for such edge 
standardised graph with maximum number of edges has the property $|\mathfrak{m}_{c}(G')|=1$.

This shows for each of the cases leads to a contradiction. Hence $|C|=1+|B|$ and $|A|=1+|C|$, i.e. $\nabla_{c}(G)=2$.
\end{proof}

\begin{proposition}\label{sum_t+1_gain-0}
Let $G$ be an edge standardised graph with lambda chromatic number $t\geq3$ and $C_{0},\ldots,C_{t}$ be the coloured partition  of the vertex set of $G$ with respect to the underlying optimal lambda colouring $c$. 
If $\nabla_{c}(G)\geq2$, $G$ contains maximum number of edges with 
$|\mathfrak{m}_{c}(G)|+|\mathfrak{M}_{c}(G)|=t+1$ and for some $A\in\mathfrak{M}_{c}(G)$ and $B\in\mathfrak{m}_{c}(G)$, 
$|E(\mathscr{I}_{B}\mathscr{D}_{A}[G])|=|E(G)|$, then $|\mathfrak{m}_{c}(G)|=1$. Consequently, the following results hold.
\begin{enumerate}[\normalfont(a)]
\item $3\leq|\mathfrak{M}_{c}(G)|\leq4$ and $3\leq t\leq4$.
\item $\mathfrak{m}_{c}(G)=\{C_{i}\}$, where $1\leq i\leq t-1$.
\item If $t=4$, then $\nabla_{c}(G)=2$.
\item If $t=3$, then $2\leq\nabla_{c}(G)\leq3$.
\end{enumerate}
\end{proposition}
\begin{proof}
If $\mathfrak{M}_{c}(G)$ contains a unique element $X$, then due to the assumption $\nabla_{c}(G)\geq2$, we 
have for each $Y\in\mathfrak{m}_{c}(G)$, $|E(\mathscr{I}_{Y}\mathscr{D}_{X}[G])|\geq1+|E(G)|$. Since both the 
graphs $\mathscr{I}_{Y}\mathscr{D}_{X}[G]$ and $G$ are edge standardised graphs, this leads to a 
contradiction to the assumption that $G$ contains maximum number of edges. Hence $|\mathfrak{M}_{c}(G)|\geq2$. 
Let for some $A\in\mathfrak{M}_{c}(G)$ and $B\in\mathfrak{m}_{c}(G)$, 
$|E(\mathscr{I}_{B}\mathscr{D}_{A}[G])|=|E(G)|$ holds. Suppose $|\mathfrak{m}_{c}(G)|\geq2$, then for the edge standardised graph $G':=\mathscr{I}_{B}\mathscr{D}_{A}[G]$ we have 
$|\mathfrak{M}_{c}(G')|+|\mathfrak{m}_{c}(G')|=t-1$. Hence we conclude by Lemma~\ref{atmostone}, $G'$ can not 
contain maximum number of edges. But $|E(G')|=|E(G)|$. This implies $G$ can not contain maximum number of edges. 
A contradiction arises. Hence $|\mathfrak{m}_{c}(G)|=1$ and the result follows.

Suppose $|\mathfrak{M}_{c}(G)|\geq5$. This, together with $|\mathfrak{m}_{c}(G)|=1$ and 
$|\mathfrak{m}_{c}(G)|+|\mathfrak{M}_{c}(G)|=t+1$, yields the existence of an integer $i$, with $0\leq i\leq t-2$ 
such that $C_{i}$, $C_{i+1}$, $C_{i+2}$ (i.e. consecutive three members) are members of 
$\mathfrak{M}_{c}(G)$, this contradicts \eqref{starstar} of Proposition~\ref{nablageq2}, therefore 
$|\mathfrak{M}_{c}(G)|\leq4$. Since $|\mathfrak{M}_{c}(G)|+|\mathfrak{m}_{c}(G)|=t+1\geq4$, we have 
$|\mathfrak{M}_{c}(G)|\geq3$. Hence the result (a) holds.

Suppose $C_{0}$ is the unique member of $\mathfrak{m}_{c}(G)$. This, together with (a), yields the existence of an integer $i$, with $1\leq i\leq t-2$ such that $C_{i}$, $C_{i+1}$, $C_{i+2}$ (i.e. consecutive three members) 
are members of $\mathfrak{M}_{c}(G)$, this contradicts \eqref{starstar} of Proposition~\ref{nablageq2}. 
Similarly our assumption, $C_{t}$ is the unique member of $\mathfrak{m}_{c}(G)$ leads to a contradiction. 
This implies the result (b).

To show (c), first we note that since $t=4$, by using (b) of Proposition~\ref{nablageq2}, $\{C_{0},C_{1},C_{3},C_{4}\}$ 
is the complete list of members of $\mathfrak{M}_{c}(G)$ and $\mathfrak{m}_{c}(G)=\{C_{2}\}$. 
Suppose $\nabla_{c}(G)\geq3$. We note that for this case
$|\mathfrak{M}_{c}(G')|=3$ and $\mathfrak{m}_{c}(G')$ contains unique member say $B'$. Therefore there exists $A'\in\mathfrak{M}_{c}(G')$ such that $|\mathsf{P}_{G'}(A')\cap\mathfrak{M}_{c}(G')|=1$. Also 
$|\mathsf{P}_{G'}(B')|=2$. Hence by using \eqref{star} of Lemma~\ref{atmostone} we have 
$|E(\mathscr{I}_{B'}\mathscr{D}_{A'}[G'])|\geq1+|E(G')|=1+|E(G)|$. This is a contradiction 
to the assumption that $G$ contains maximum number of edges. Hence  $\nabla_{c}(G)=2$.

To show (d), first we note that since $t=3$, by using (b) of Proposition~\ref{nablageq2}, either 
$\{C_{0},C_{2},C_{3}\}$ is the complete list of members of  $\mathfrak{M}_{c}(G)$ and $\mathfrak{m}_{c}(G)=\{C_{1}\}$ 
or $\{C_{0},C_{1},C_{3}\}$ is the complete list of members of 
$\mathfrak{M}_{c}(G)$ and $\mathfrak{m}_{c}(G)=\{C_{2}\}$. Note that for the first choice and second choice of the coloured 
partitions, we have $|E(G')|=|E(G)|$ implies $A\neq C_{0}$ and $A\neq C_{3}$ for the respective choices. 
Suppose $\nabla_{c}(G)\geq4$. We note that for this case $|\mathfrak{M}_{c}(G')|=2$ and $\mathfrak{m}_{c}(G')$ contains unique member say $B'$. Let $A'\in\mathfrak{M}_{c}(G')$ and we construct $G'':=\mathscr{I}_{B'}\mathscr{D}_{A'}[G']$. Here 
$|\mathfrak{M}_{c}(G'')|=1=|\mathfrak{m}_{c}(G'')|$. We assume $\mathfrak{M}_{c}(G'')=\{X\}$ and $\mathfrak{m}_{c}(G'')=\{Y\}$, then $|\mathfrak{M}_{c}(G'')|+|\mathfrak{m}_{c}(G'')|=2=t-1$. Hence we conclude by Lemma~\ref{atmostone}, $G''$ can not 
contain maximum number of edges. But $|E(G'')|=|E(G')|=|E(G)|$. This contradicts the assumption that $G$ contains maximum number of edges. Hence  $2\leq\nabla_{c}(G)\leq3$.  
\end{proof}

\begin{proposition}\label{sum_t+1_gain-1}
Let $G$ be an edge standardised graph with lambda chromatic number $t\geq3$ and $C_{0},\ldots,C_{t}$ be the coloured partition  of the vertex set of $G$ with respect to the underlying optimal lambda colouring $c$. If 
$\nabla_{c}(G)\geq2$, $G$ contains maximum number of edges and $|E(\mathscr{I}_{B}\mathscr{D}_{A}[G])|=|E(G)|-1$ for all $A\in\mathfrak{M}_{c}(G)$ and $B\in\mathfrak{m}_{c}(G)$, then $|\mathfrak{m}_{c}(G)|=2$,
$|\mathfrak{M}_{c}(G)|=3$, $t=4$ and $\nabla_{c}(G)=2$. Moreover, if $C_{i}\in\mathfrak{m}_{c}(G)$, then $1\leq i\leq t-1$.
\end{proposition}
\begin{proof}
Here we observe from \eqref{star} of Lemma~\ref{atmostone}, that for all $A\in\mathfrak{M}_{c}(G)$ and 
$B\in\mathfrak{m}_{c}(G)$, $|E(\mathscr{I}_{B}\mathscr{D}_{A}[G])|=|E(G)|-1$ if and only if 
$|\mathfrak{M}_{c}(G)\cap\mathsf{P}_{G}(A)|=0$, $|\mathfrak{m}_{c}(G)\cap\mathsf{P}_{G}(B)|=0$ for all 
$A\in\mathfrak{M}_{c}(G)$ and $B\in\mathfrak{m}_{c}(G)$, $|\mathfrak{M}_{c}(G)|+|\mathfrak{m}_{c}(G)|=t+1$ and 
$|\mathsf{P}_{G}(B)|=2$ for all $B\in\mathfrak{m}_{c}(G)$, i.e. if $C_{i}\in\mathfrak{m}_{c}(G)$, then $1\leq i\leq t-1$.   

Suppose $|\mathfrak{m}_{c}(G)|\geq3$. Since $|\mathfrak{M}_{c}(G)\cap\mathsf{P}_{G}(A)|=0$, 
$|\mathfrak{m}_{c}(G)\cap\mathsf{P}_{G}(B)|=0$ and $|\mathsf{P}_{G}(B)|=2$ for all 
$A\in\mathfrak{M}_{c}(G)$ and $B\in\mathfrak{m}_{c}(G)$, we have $|\mathfrak{M}_{c}(G)|=|\mathfrak{m}_{c}(G)|+1\geq4$. We choose $X\in\mathfrak{M}_{c}(G)$ and $Y\in\mathfrak{m}_{c}(G)$ and construct the graph $G':=\mathscr{I}_{Y}\mathscr{D}_{X}[G]$, then by assumption $|E(G')|=|E(G)|-1$. Also $|\mathfrak{m}_{c}(G')|=|\mathfrak{m}_{c}(G)|-1\geq2$ and 
$|\mathfrak{M}_{c}(G')|=|\mathfrak{M}_{c}(G)|-1\geq3$. Therefore $|\mathfrak{M}_{c}(G')|+|\mathfrak{m}_{c}(G')|=t-1$. We note that $|\mathfrak{M}_{c}(G')\cap\mathsf{P}_{G'}(A')|=0$, $|\mathfrak{m}_{c}(G')\cap\mathsf{P}_{G'}(B')|=0$ and 
$|\mathsf{P}_{G'}(B')|=2$, for all $A'\in\mathfrak{M}_{c}(G')$ and $B'\in\mathfrak{m}_{c}(G')$. Now we choose 
$X'\in\mathfrak{M}_{c}(G')$ and $Y'\in\mathfrak{m}_{c}(G')$ and construct the graph
$G'':=\mathscr{I}_{Y'}\mathscr{D}_{X'}[G']$. Therefore $|\mathfrak{M}_{c}(G'')|=|\mathfrak{M}_{c}(G')|-1\geq2$, 
$|\mathfrak{m}_{c}(G'')|=|\mathfrak{m}_{c}(G')|-1\geq1$ and $|\mathfrak{M}_{c}(G'')|+|\mathfrak{m}_{c}(G'')|=t-3$. Hence we conclude by Lemma~\ref{atmostone}, $G''$ can not contain maximum number of edges. But using \eqref{star} of Lemma~\ref{atmostone} we have $|E(G'')|=|E(G')|+1=E(G)$. This contradicts the assumption that $G$ contains maximum number of edges. Hence  $|\mathfrak{m}_{c}(G)|\leq2$. Now suppose $|\mathfrak{m}_{c}(G)|=1$. Since 
$|\mathfrak{M}_{c}(G)\cap\mathsf{P}_{G}(A)|=0$, $|\mathfrak{m}_{c}(G)\cap\mathsf{P}_{G}(B)|=0$ and 
$|\mathsf{P}_{G}(B)|=2$ for all $A\in\mathfrak{M}_{c}(G)$ and $B\in\mathfrak{m}_{c}(G)$, therefore 
we have $|\mathfrak{M}_{c}(G)|=2$, i.e. $t=2$. A contradiction arises since $t\geq3$. Hence 
$|\mathfrak{m}_{c}(G)|=2$. Consequently, $|\mathfrak{M}_{c}(G)|=3$ and $t=4$.

Suppose $\nabla_{c}(G)\geq3$. We choose $X\in\mathfrak{M}_{c}(G)$ and $Y\in\mathfrak{m}_{c}(G)$ and construct the graph $G':=\mathscr{I}_{Y}\mathscr{D}_{X}[G]$, then by assumption $|E(G')|=|E(G)|-1$. We note that, here 
$|\mathfrak{m}_{c}(G')|=|\mathfrak{m}_{c}(G)|-1=1$ and $|\mathfrak{M}_{c}(G')|=|\mathfrak{M}_{c}(G)|-1=2$. We choose $X'\in\mathfrak{M}_{c}(G')$ and $Y'\in\mathfrak{m}_{c}(G')$ and construct the graph 
$G'':=\mathscr{I}_{Y'}\mathscr{D}_{X'}[G']$. Therefore $|\mathfrak{M}_{c}(G'')|=|\mathfrak{M}_{c}(G')|-1=1$, 
$|\mathfrak{m}_{c}(G'')|=|\mathfrak{m}_{c}(G')|+1=2$ and $|\mathfrak{M}_{c}(G'')|+|\mathfrak{m}_{c}(G'')|=3=t-1$. Hence we conclude by Lemma~\ref{atmostone}, $G''$ can not contain maximum number of edges. Since 
$|\mathfrak{M}_{c}(G')\cap\mathsf{P}_{G'}(A')|=0$, $|\mathfrak{m}_{c}(G')\cap\mathsf{P}_{G'}(B')|=0$ and 
$|\mathsf{P}_{G'}(B')|=2$, for all $A'\in\mathfrak{M}_{c}(G')$ and $B'\in\mathfrak{m}_{c}(G')$, by using  \eqref{star} of Lemma~\ref{atmostone} we have $|E(G'')|=|E(G')|+1=E(G)$. This contradicts the assumption that $G$ contains maximum number of edges. Hence $\nabla_{c}(G)=2$.
\end{proof}

\section{The final arc : From Stationary Results to The Classification Results}

In this section, our main aim is to establish a classification result.  We classify all $n-$vertex graphs with lambda 
chromatic number $t\geq3$ and $n\geq t+1$.  

\begin{proposition}
Let $G$ be an graph with lambda chromatic number $t$ and $C_{0},\ldots,C_{t}$ be the coloured partition  of the vertex set of $G$ with respect to an optimal lambda colouring $c$. Then the mapping $\bar{c}:V(G)\rightarrow\{0,\ldots,t\}$ defined by $u\mapsto t-c(u)$ is an optimal lambda colouring of $G$ and 
$\bar{C}_{0},\ldots,\bar{C}_{t}$, where $\bar{C}_{i}=C_{t-i}$ for each integer $i$ with $0\leq i\leq t$, is the coloured partition  of the vertex set of $G$ with respect to $\bar{c}$.
\end{proposition}
\begin{proof}
The result follows immediately since $|c(u)-c(v)|+\di[G]{u,v}=|\bar{c}(u)-\bar{c}(v)|+\di[G]{u,v}$ for all $u,v\in V(G)$. 
\end{proof}

\begin{definition}
The aforementioned $\bar{c}: V(G)\rightarrow\{0,\ldots,t\}$ is called the \emph{dual} of the optimal lambda colouring $c$. The coloured partition $\bar{C}_{0},\ldots,\bar{C}_{t}$ is called the \emph{dual coloured partition} of  
$C_{0},\ldots,C_{t}$. 
\end{definition}

The following five results are most important ingredient to prove the classification results (Theorem~\ref{classification} and Corollary~\ref{classification-large}). Here we fix the notation 
\begin{equation*}
\K[G]{U_{0},\ldots,U_{t}}:=|\left\{(i,i+1):|U_{i}|=|U_{i+1}|=\max\{|U_{n}|:0\leq n\leq t\}\right\}|,
\end{equation*}
for a partition $U_{0},\ldots,U_{t}$ of the vertex set of $G$. Therefore $\K[G]{U_{0},\ldots,U_{t}}$ counts the number pairs of the form $(U_{i},U_{i+1})$, where $0\leq i\leq t-1$ (i.e. the \emph{consecutive pairs}) and 
$|U_{i}|=|U_{i+1}|=\underset{0\leq n\leq t}{\max}|U_{n}|$. 

\begin{theorem}\label{nabla_1}
Let $G$ be an edge standardised graph with lambda chromatic number $t\geq3$ and $C_{0},\ldots,C_{t}$ be the coloured partition  of the vertex set of $G$ with respect to the underlying optimal lambda colouring $c$. If $G$ contains $n$ vertices and $\nabla_{c}(G)=1$, then $n=|B|(t+1)+|\mathfrak{M}_{c}(G)|$ and 
\begin{equation*}
|E(G)|=|B|\binom{t}{2}+\binom{|\mathfrak{M}_{c}(G)|}{2}-\K[G]{C_{0},\ldots,C_{t}},
\end{equation*}
where $B\in\mathfrak{m}_{c}(G)$. Moreover, $|B|=\lfloor\frac{n}{t+1}\rfloor$ and 
$|\mathfrak{M}_{c}(G)|=n-(t+1)\lfloor\frac{n}{t+1}\rfloor$. Also $G$ contains maximum number of edges if and only if 
the value $\K[G]{C_{0},\ldots,C_{t}}$ is minimum over any equitable partition into (unequally sized) $t+1$ parts (subsets) of the vertex sets of all possible $n-$vertex graphs with lambda chromatic number $t$.
\end{theorem}
\begin{proof}
Let $B\in\mathfrak{m}_{c}(G)$. Here the vertex set of the edge standardised graph $G$ admits a partition into two parts (subsets) $U_{1}$ and $U_{2}$ such that $\ed[G]{U_{1},U_{2}}=0$, where
\begin{equation*}
U_{1}:=\left(\overset{t}{\underset{m=0}\sqcup}\{u^{m}_{i}:1\leq i\leq|B|\}\right)\textup{ and }
U_{2}:=\left(\underset{C_{m}\in\mathfrak{M}_{c}(G)}{\sqcup}\{u^{m}_{i}:i=|B|+1\}\right).
\end{equation*}
Hence the formula for $n$ holds. And to count the edges it is enough to count the edges of the subgraphs induced 
by the subsets $U_{1}$ and $U_{2}$. Since $G$ is an edge standardised graph, the subgraph induced by the subset 
$U_{1}$ of $V(G)$ is $|B|$ disjoint copies of $\mathbb{G}_{t}$. Therefore such subset yields $|B|\binom{t}{2}$ edges. To calculate the remaining edges we note that $G$ is an edge standardised graph. So, the subgraph induced by the 
subset $U_{2}$ of $V(G)$ is the \emph{almost complete graph} on $|\mathfrak{M}_{c}(G)|$ vertices i.e complete 
graph on $|\mathfrak{M}_{c}(G)|$ vertices with $\K[G]{C_{0},\ldots,C_{t}}$ deleted edges. Therefore there are 
$\left[\binom{|\mathfrak{M}_{c}(G)|}{2}-\K[G]{C_{0},\ldots,C_{t}}\right]$ additional edges.

For the next part we note that $\nabla_{c}(G)=1$ implies 
$1\leq|\mathfrak{M}_{c}(G)|\leq t\Leftrightarrow\frac{1}{t+1}\leq \frac{n}{t+1}-|B|\leq\frac{t}{t+1}$
Consequently, $|B|=\lfloor\frac{n}{t+1}\rfloor$ and $|\mathfrak{M}_{c}(G)|=n-(t+1)\lfloor\frac{n}{t+1}\rfloor$. The remaining part follows immediately from the edge formula.  
\end{proof}

\begin{theorem}\label{t=3}
Let $G$ be an edge standardised graph with lambda chromatic number $3$ and $C_{0},C_{1},C_{2},C_{3}$ be the coloured partition of the vertex set of $G$ with respect to the underlying optimal lambda colouring $c$. Then $G$ contains maximum number of edges and $\nabla_{c}(G)\geq2$ if and only if the coloured partition or its  respective dual partition satisfies one of the following four types of properties.
\begin{enumerate}[\normalfont(a)]
\item $\mathfrak{M}_{c}(G)=\{C_{0},C_{3}\}$, $\mathfrak{m}_{c}(G)=\{C_{2}\}$ and 
$|C_{0}|=|C_{1}|+1=|C_{2}|+2=|C_{3}|$.
\item $\mathfrak{M}_{c}(G)=\{C_{1},C_{3}\}$, $\mathfrak{m}_{c}(G)=\{C_{2}\}$ and 
$|C_{0}|+1=|C_{1}|=|C_{2}|+2=|C_{3}|$.
\item $\mathfrak{M}_{c}(G)=\{C_{0},C_{2},C_{3}\}$, $\mathfrak{m}_{c}(G)=\{C_{1}\}$ and 
$|C_{0}|=|C_{1}|+2=|C_{2}|=|C_{3}|$.
\item $\mathfrak{M}_{c}(G)=\{C_{0},C_{2},C_{3}\}$, $\mathfrak{m}_{c}(G)=\{C_{1}\}$ and 
$|C_{0}|=|C_{1}|+3=|C_{2}|=|C_{3}|$.
\end{enumerate}
\end{theorem}
\begin{proof}
If $G$ contains maximum number of edges, $|\mathfrak{M}_{c}(G)|+|\mathfrak{m}_{c}(G)|=t=3$ and 
$\nabla_{c}(G)\geq2$, then conclusion (a) and (b) directly follow from Proposition~\ref{sum_t} and Proposition~\ref{nablageq2}. For the converse part of (a), let $G$ be the same as mentioned in the statement with additional property that (a) holds. Then we construct, $G':=\mathscr{I}_{C_{2}}\mathscr{D}_{C_{3}}[G]$. Here $\nabla_{c}(G')=1$ and $|\mathfrak{M}_{c}(G')|=1$. Hence $\K[G']{C'_{0},C'_{1},C'_{2},C'_{3}}=0$, where $C'_{0},C'_{1},C'_{2},C'_{3}$ is the transformed (coloured) partition of $C_{0},C_{1},C_{2},C_{3}$. 
Therefore using Theorem~\ref{nabla_1}, we have $G'$ contains maximum number of edges. We note that $|E(G)|=|E(G')|$. 
Hence $G$ contains maximum number of edges. For the converse part of (b), we assume $G$ is the same as 
mentioned in the statement with additional property that (b) holds. Then with a same construction and similar 
argument we have such $G$ contains maximum number of edges.

Similarly if $G$ contains maximum number of edges, $|\mathfrak{M}_{c}(G)|+|\mathfrak{m}_{c}(G)|=t+1=4$ and $\nabla_{c}(G)\geq2$, then conclusion (c) and (d) directly follow from Proposition~\ref{sum_t+1_gain-0}. Conversely let $G$ be the same as mentioned in the statement with additional property that (c) holds. Then $\nabla_{c}(G)=2$ and we let $G':=\mathscr{I}_{C_{1}}\mathscr{D}_{C_{3}}[G]$. Using Theorem~\ref{nabla_1} and a similar argument as before we have $G$ contains maximum number of edges. Conversely let $G$ be the same as mentioned in the statement with additional property that (d) holds. Then $\nabla_{c}(G)=3$ and we let 
$G':=\mathscr{I}_{C_{1}}\mathscr{D}_{C_{2}}\mathscr{I}_{C_{1}}\mathscr{D}_{C_{3}}[G]$. 
Using Theorem~\ref{nabla_1} and a similar argument as before we have $G$ contains maximum number of edges.
\end{proof}

\begin{theorem}\label{t=4necessary}
Let $G$ be an edge standardised graph with lambda chromatic number $4$ and $C_{0},\ldots,C_{4}$ be the coloured partition of the vertex set of $G$ with respect to the underlying optimal lambda colouring $c$. If $G$ contains maximum number of edges and $\nabla_{c}(G)\geq2$ then there are only following three types of coloured partitions 
or their respective dual coloured partitions.
\begin{enumerate}[\normalfont(a)]
\item $\mathfrak{M}_{c}(G)=\{C_{0},C_{2},C_{4}\}$, $\mathfrak{m}_{c}(G)=\{C_{1}\}$ and 
$|C_{0}|=|C_{1}|+2=|C_{2}|=|C_{3}|+1=|C_{4}|$.
\item $\mathfrak{M}_{c}(G)=\{C_{0},C_{2},C_{4}\}$, $\mathfrak{m}_{c}(G)=\{C_{1},C_{3}\}$ and 
$|C_{0}|=|C_{1}|+2=|C_{2}|=|C_{3}|+2=|C_{4}|$.
\item $\mathfrak{M}_{c}(G)=\{C_{0},C_{1},C_{3},C_{4}\}$, $\mathfrak{m}_{c}(G)=\{C_{2}\}$ and 
$|C_{0}|=|C_{1}|=|C_{2}|+2=|C_{3}|=|C_{4}|$.
\end{enumerate}
\end{theorem}
\begin{proof}
If $G$ contains maximum number of edges, $|\mathfrak{M}_{c}(G)|+|\mathfrak{m}_{c}(G)|=t=4$ and 
$\nabla_{c}(G)\geq2$, then the conclusion (a) directly follows from Proposition~\ref{sum_t} and Proposition~\ref{nablageq2}.

If $G$ contains maximum number of edges, $|\mathfrak{M}_{c}(G)|+|\mathfrak{m}_{c}(G)|=t+1=5$ and 
$\nabla_{c}(G)\geq2$, then the conclusion (b) and (c) directly follow from Proposition~\ref{sum_t+1_gain-0}, Proposition~\ref{sum_t+1_gain-1} and Proposition~\ref{nablageq2}.  
\end{proof}

\begin{theorem}\label{t=4sufficint}
Let $G$ be an edge standardised graph with lambda chromatic number $4$ and $C_{0},\ldots,C_{4}$ be the coloured partition of the vertex set of $G$ with respect to the underlying optimal lambda colouring $c$. 
If the coloured partition or its respective dual coloured partition satisfies one of the following two properties, 
then $G$ contains maximum number of edges.
\begin{enumerate}[\normalfont(a)]
\item $\mathfrak{M}_{c}(G)=\{C_{0},C_{2},C_{4}\}$, $\mathfrak{m}_{c}(G)=\{C_{1}\}$ and 
$|C_{0}|=|C_{1}|+2=|C_{2}|=|C_{3}|+1=|C_{4}|$.
\item $\mathfrak{M}_{c}(G)=\{C_{0},C_{2},C_{4}\}$, $\mathfrak{m}_{c}(G)=\{C_{1},C_{3}\}$ and 
$|C_{0}|=|C_{1}|+2=|C_{2}|=|C_{3}|+2=|C_{4}|$.
\end{enumerate}
\end{theorem}
\begin{proof}
We assume (a) holds. Let $G':=\mathscr{I}_{C_{1}}\mathscr{D}_{C_{4}}[G]$. Here $\nabla_{c}(G')=1$ and 
$|\mathfrak{M}_{c}(G')|=2$. Also \\$\K[G']{C'_{0},\ldots,C'_{4}}=0$, where $C'_{0},\ldots,C'_{4}$ is the 
transformed (coloured) partition of $C_{0},\ldots,C_{4}$. Therefore using Theorem~\ref{nabla_1}, we have 
$G'$ contains maximum number of edges. We note that $|E(G)|=|E(G')|$. Hence such $G$ contains maximum number 
of edges and the result follow for this case.

We assume (b) holds. Let $G':=\mathscr{I}_{C_{3}}\mathscr{D}_{C_{2}}\mathscr{I}_{C_{1}}\mathscr{D}_{C_{4}}[G]$. 
Here $\nabla_{c}(G')=1$ and $|\mathfrak{M}_{c}(G')|=1$. Hence $\K[G']{C'_{0},\ldots,C'_{4}}=0$, where 
$C'_{0},\ldots,C'_{4}$ is the transformed (coloured) partition of $C_{0},\ldots,C_{4}$. Therefore using 
Theorem~\ref{nabla_1}, we have $G'$ contains maximum number of edges. We note that $|E(G)|=|E(G')|$. Hence such 
$G$ contains maximum number of edges and the result follow for this case.
\end{proof}

\begin{remark}
The converse part of Theorem~\ref{t=4necessary} is not true. Suppose $G$ is a graph, which satisfies the 
conditions mention in the statement of Theorem~\ref{t=4necessary} and in addition (c) holds, then $G$ 
does not contain maximum number of edges. 

Technically, it is due to for each $X\in\mathfrak{M}_{c}(G)$,  $G':=\mathscr{I}_{C_{2}}\mathscr{D}_{X}[G]$ implies 
$|\mathfrak{M}_{c}(G')|=3$ and there exists exactly one integer $i$, with $i=0$ or $i=3$, such that $\{C'_{i},C'_{i+1}\}\subset\mathfrak{M}_{c}(G')$. Hence $\K[G']{C'_{0},\ldots,C'_{4}}=1$.
Suppose $H$ is a graph with lambda chromatic number $t=4$, $\nabla_{c}(H)=1$, 
$|\mathfrak{M}_{c}(H)|=3$ with respect to the underlying optimal lambda colouring $c:V(H)\rightarrow\{0,\ldots,4\}$. 
Then by using Theorem~\ref{nabla_1}, such $H$ contains maximum number of edges if and only if 
the value of $\K[H]{C_{0},\ldots,C_{4}}$ is $0$. Here $G'$ satisfies $t=4$, $\nabla_{c}(G')=1$, 
$|\mathfrak{M}_{c}(G')|=3$ but $\K[G']{C'_{0},\ldots,C'_{4}}=1$. Therefore such $G'$ can not contain 
maximum number of edges. We note that $|E(G)|=|E(G')|$. Hence such $G$ can not contain maximum number of edges.

Therefore any edge standardised graph satisfying the (stationary) conditions mentioned in (c) of Theorem~\ref{t=4necessary}  does not contain maximum number of edges. 
\end{remark}

To propose the classification results, we can not restrict ourselves only to edge standardised graphs. So we need to get rid of the ``edge distribution'' related restriction that makes a graph edge standardised. Henceforth, the graphs 
are not necessarily edge standardised. 

\begin{theorem}\label{t+1_divides_n}
Let $G$ be a graph with lambda chromatic number $t\geq3$ and $C_{0},\ldots,C_{t}$ be the coloured partition  of the vertex set of $G$ with respect to an optimal lambda colouring $c$. If $G$ contains $n\geq t+1$ vertices, where $n$ is a multiple of $t+1$, then $G$ contains maximum number of edges if and only if 
$G$ is a member graph of $\mathsf{G}(t,\frac{n}{t+1})$. 
\end{theorem}
\begin{proof}
Let $G$ be a member graph of $\mathsf{G}(t,\frac{n}{t+1})$. Then $G$ has $n$ vertices and its lambda chromatic 
number is $t$. It follows from Theorem~\ref{universality}, that members of $\mathsf{G}(t,l)$, for some integer $l$, has maximum number of edges among the graphs with lambda chromatic number $t$ and at most $l(t+1)$ vertices. Hence,
for $l=\frac{n}{t+1}$, $G$ has maximum number of edges. 

Conversely, suppose $G$ has maximum number of edges and $c:V(G)\rightarrow\{0,\ldots,t\}$ is an optimal lambda colouring of $G$. If $\nabla_{c}(G)=1$, then $(t+1)$ does not 
divide $n$. A contradiction arises. Now if $\nabla_{c}(G)\geq2$, then we construct $\mathscr{S}_{c}[G]$. From the definition of edge standardised graph, we have $V(\mathscr{S}_{c}[G])=V(G)$. 
By using Proposition~\ref{es_base}, we have the lambda chromatic number of $\mathscr{S}_{c}[G]$ is $t$. 
Since $G$ contains maximum number of edges, again by using Proposition~\ref{es_base} we have $|E(\mathscr{S}_{c}[G])|=|E(G)|$. This means the edge standardised graph $\mathscr{S}_{c}[G]$ contains maximum number of edges. Therefore by using Theorem~\ref{t=3} and Theorem~\ref{t=4sufficint}, we have $\nabla_{c}(G)=\nabla_{c}(\mathscr{S}_{c}[G])=2\textup{ or }3$. 
If $\nabla_{c}(G)=\nabla_{c}(\mathscr{S}_{c}[G])=2\textup{ or }3$, then it follows from Theorem~\ref{t=3} and Theorem~\ref{t=4sufficint} that $n=|V(\mathscr{S}_{c}[G])|=|V(G)|$ is of the form $s(t+1)+r$, where $s\geq0$ and $r\geq1$ are integers. Here $r$ equals $5\mod{4}$, $5\mod{4}$, $6\mod{4}$, $9\mod{4}$, $7\mod{5}$ and $6\mod{5}$ for the respective cases. Hence we conclude $(t+1)$ does not divides $n$. This leads to a contradiction. Therefore  $\nabla_{c}(G)=0$ and consequently $G$ is a member graph of $\mathsf{G}(t,\frac{n}{t+1})$.
\end{proof}

We now define a stationary graph. Such graphs behave in conformity with the necessary conditions, 
developed in Theorem~\ref{nabla_1}, Theorem~\ref{t=3}, Theorem~\ref{t=4necessary} and Theorem~\ref{t+1_divides_n}. 
The intrinsic local restrictions of an optimal lambda colouring are also maintained.

\begin{definition}
An $n-$vertex graph $G$ is said to be a \emph{stationary graph} if the vertex set is partitioned into $t+1$ subsets
$V_{0},\ldots,V_{t}$, where $n\geq t+1\geq4$, and the edge distribution follows the following four properties.
\begin{itemize}
\item $V_{0}$ and $V_{t}$ are non-empty and if for some integer $i$, with $1\leq i\leq t-1$, $V_{i}$ is empty then both $V_{i-1}$ and $V_{i+1}$ are non-empty.
\item $\ed[G]{V_{i},V_{i}}=0$ for each integer $i$, with $0\leq i\leq t$.
\item $\ed[G]{V_{i},V_{i+1}}=0$ for each integer $i$, with $0\leq i\leq t-1$.
\item If $0<|V_{i}|\leq |V_{j}|$, where $i$ and $j$ are integers with $0\leq i,j\leq t$ and $|i-j|\geq2$, then for each $x\in V_{i}$ 
there exists a unique $y\in V_{j}$ such that $\{x,y\}$ is an edge. 
(So $\ed[G]{V_{i},V_{j}}=|V_{i}|=\min\{|V_{i}|,|V_{j}|\}$.) 
\end{itemize}
Also such partition or its dual partition (The \emph{dual partition} of $V_{0},\ldots,V_{t}$ is $\bar{V}_{0},\ldots,\bar{V}_{t}$, where $\bar{V}_{i}:=V_{t-i}$ for each integer $i$ with $0\leq i\leq t$.) satisfies exactly one of following properties.
\begin{enumerate}[\normalfont(a)]
\item $||V_{i}|-|V_{j}||\leq1$ for all integers $i$ and $j$, with $0\leq i,j\leq t$.
\item $|V_{0}|=|V_{1}|+1=|V_{2}|+2=|V_{3}|$, whenever $t=3$.
\item $|V_{0}|+1=|V_{1}|=|V_{2}|+2=|V_{3}|$, whenever $t=3$.
\item $|V_{0}|=|V_{1}|+2=|V_{2}|=|V_{3}|$, whenever $t=3$.
\item $|V_{0}|=|V_{1}|+3=|V_{2}|=|V_{3}|$, whenever $t=3$.
\item $|V_{0}|=|V_{1}|+2=|V_{2}|=|V_{3}|+1=|V_{4}|$, whenever $t=4$.
\item $|V_{0}|=|V_{1}|+2=|V_{2}|=|V_{3}|+2=|V_{4}|$, whenever $t=4$.
\item $|V_{0}|=|V_{1}|=|V_{2}|+2=|V_{3}|=|V_{4}|$, whenever $t=4$.
\end{enumerate}
\end{definition}

\begin{lemma}
Let $G$ be an $n-$vertex stationary graph. Then lambda chromatic number of $G$ is at most $t$. Moreover, for each integer $i$, 
with $0\leq i\leq t$, and for each $v\in V_{i}$, $v\mapsto i$ is a lambda colouring of $G$. 
\end{lemma}
\begin{proof}
To show the lambda chromatic number of $G$ is at most $t$, it is enough to establish the mapping mentioned in the statement (say) $c$ is a lambda colouring of $G$. 

If $\di[G]{u,v}\geq3$, then $|c(u)-c(v)|+\di[G]{u,v}\geq3$. If for some $u\in V_{i}$ and $v\in V_{j}$, where 
$0\leq i,j\leq t$, suppose $\di[G]{u,v}=1$, then $j\neq i$ since $\ed[G]{V_{i},V_{i}}=0$ for each integer $i$, 
with $0\leq i\leq t$. Also $\ed[G]{V_{i},V_{i+1}}=0$ for each integer $i$, with $0\leq i\leq t-1$. Hence 
$|c(u)-c(v)|=|i-j|\geq2$. Consequently, $|c(u)-c(v)|+\di[G]{u,v}\geq3$ whenever $\di[G]{u,v}=1$ for some $u,v\in V(G)$.
If for some $u,v\in V(G)$, suppose $\di[G]{u,v}=2$, then $u$ and $v$ can not belong to same $V_{i}$, where $0\leq i\leq t$. Otherwise there would exist $w\in V_{j}$, where $j\neq i$, such that $\{u,w\}$ and $\{v,w\}$ are edges. This is a 
contradiction. Hence $c(u)\neq c(v)$, i.e. $|c(u)-c(v)|+\di[G]{u,v}\geq3$ whenever $\di[G]{u,v}=2$ for some $u,v\in V(G)$.
\end{proof}

\begin{lemma}\label{stationary_sufficient}
Let $G$ be an $n-$vertex graph with lambda chromatic number $t\geq3$ and $n\geq t+1$. If $G$ contains maximum number of edges then $G$ is a stationary graph.
\end{lemma}
\begin{proof}
Let $c:V(G)\rightarrow\{0,\ldots,t\}$ be an optimal lambda colouring of $G$ and $C_{0},\ldots,C_{t}$ be the corresponding coloured partition of $V(G)$. This partition of $V(G)$ satisfies the edge distribution related 
properties in the definition of a stationary graph. By Proposition~\ref{es_base}, $c$ is also an optimal lambda colouring of $\mathscr{S}_{c}[G]$. We note that if $G$ contains maximum number of edges then $|E(G)|=|E(\mathscr{S}_{c}[G])|$. From the definition of edge standardised graph, we have $V(\mathscr{S}_{c}[G])=V(G)$. Moreover, using Proposition~\ref{es_base}, we have lambda chromatic number of $\mathscr{S}_{c}[G]$ is $t$. Hence $\mathscr{S}_{c}[G]$ contains maximum number of edges. Also note that 
$C_{i}=\{u\in V(\mathscr{S}_{c}[G]):c(u)=i\}$ for each integer $i$, with $0\leq i\leq t$.

Now $\nabla_{c}(\mathscr{S}_{c}[G])=\nabla_{c}(G)$. If $\nabla_{c}(G)\geq2$, then using Theorem~\ref{t=3} and Theorem~\ref{t=4necessary}, the (coloured) partition $C_{0},\ldots,C_{t}$ of the vertex set of 
$\mathscr{S}_{c}[G]$ (and hence $G$) follows exactly one of the conditions stated from (b) to (h), in the definition of stationary graph. If $\nabla_{c}(G)\leq1$, then the vertex partition $C_{0},\ldots,C_{t}$ of vertex set follows condition (a) in the definition of stationary graph. 
\end{proof}

Now we are in a position to establish our final classification results. This concludes our article. 

\begin{theorem}\label{classification}
Let $G$ be an $n-$vertex graph with lambda chromatic number $t\geq3$ and $n\geq t+1$. Then $G$ contains maximum number of edges if and only if exactly one of the following holds.
\begin{enumerate}[\normalfont(i)]
\item $G$ is isomorphic to an $n-$vertex member graph $G^{*}$ of $\mathsf{G}(t,\frac{n}{t+1})$, where $n\equiv0\mod{t+1}$.
\item $G$ is isomorphic to an $n-$vertex stationary graph satisfying the property {\normalfont(a)} 
(mentioned in the  definition), such that the value $\K[G]{V_{0},\ldots,V_{t}}$ is minimum over any 
equitable partition into $t+1$ unequally sized parts (subsets) of the vertex sets of all possible $n-$vertex 
graphs, where $n\not\equiv0\mod{t+1}$. 
\item $G$ is isomorphic to exactly one of the $n-$vertex stationary graph satisfying the property {\normalfont(b)} to {\normalfont(g)} (mentioned in the definition), where $n\not\equiv0\mod{t+1}$.
\end{enumerate}
\end{theorem}
\begin{proof}
If $n\equiv0\mod{t+1}$, then from Theorem~\ref{t+1_divides_n}, $G$ has maximum number of edges if and only if (i) holds.

Suppose $n\not\equiv0\mod{t+1}$. Let (ii) hold. Then $V_{0},\ldots,V_{t}$ is a partition of $V(G)$ and there exist integers 
$i,j$, with $0\leq i,j\leq t$, such that $|V_{i}|=1+|V_{j}|$. Since the graph $G$ is a stationary graph, the colouring
$c:V(G)\rightarrow\{0,\ldots,t\}$ defined by $c(v)=i$, where $v\in V_{i}$, $0\leq i\leq t$, is an optimal lambda colouring of $G$. Clearly, $V_{0},\ldots,V_{t}$ is the underlying coloured partition of the edge standardised graph $\mathscr{S}_{c}[G]$. Now $\ed[G]{V_{i}, V_{j}}=\ed[{\mathscr{S}_{c}[G]}]{V_{i},V_{j}}$, $0\leq i,j\leq t$. Hence $|E(G)|=|E(\mathscr{S}_{c}[G])|$. Therefore using Theorem~\ref{nabla_1}, $\mathscr{S}_{c}[G]$ and hence $G$ has maximum number of edges.

Let (iii) hold. Then using similar argument as above, Theorem~\ref{t=3} and Theorem~\ref{t=4sufficint}, $\mathscr{S}_{c}[G]$ and hence $G$ has maximum number of edges.

Conversely, suppose $G$ has maximum number of edges. Then by Lemma~\ref{stationary_sufficient}, $G$ is a stationary graph. Therefore $G$ satisfies exactly one condition from (a) to (h) stated in the definition of stationary graph. But condition (h) is further excluded by the aforementioned remark. Further $n\not\equiv0\mod{t+1}$ excludes the fact (i) of this hypothesis.
Hence facts of (ii) and (iii) of hypothesis follow.
\end{proof} 

The following two results are natural consequences of the above classification theorem.

\begin{corollary}\label{classification-large}
Let $G$ be an $n-$vertex graph with lambda chromatic number $t\geq5$ and $n\geq t+1$. If $G$ contains maximum number of edges, then  $G$ admits an optimal equitable partition.
\end{corollary}

The converse of the above corollary is not true. However, if an $n-$vertex stationary graph $G$ with equitable partition $V_{0},\ldots,V_{t}$ satisfies the property that the value $\K[G]{V_{0},\ldots,V_{t}}$ is minimum over any equitable partition into $t+1$ parts (subsets) of the vertex sets of all possible $n-$vertex graphs, then $G$ contains maximum number of edges.  

\begin{corollary}
Let $G$ be an $n-$vertex graph with lambda chromatic number $t\geq3$ and $n\geq t+1$. If $G$ contains maximum number of edges then there exist a member graph $G^{*}$ of $\mathsf{G}(t,\lfloor\frac{n}{t+1}\rfloor)$ and  a member graph $G^{**}$ of 
$\mathsf{G}(t,\lfloor\frac{n}{t+1}\rfloor+3)$ such that $G^{*}$ is subgraph of $G$ and $G$ is subgraph of $G^{**}$.
\end{corollary}

The above corollary connotes an approximation result. Roughly, an $n-$vertex graph with lambda chromatic number
$t\geq3$, where $n\geq t+1$, and having maximum number of edges can be approximated by an ``inner'' graph $G^{*}$ 
and an ``outer'' graph $G^{**}$.

\begin{acknowledgement}
The research work of first author is supported by the post doctoral fellowship scheme (File Reference Number: 2/40(33)/2015/R\&D-II/11174 dated August~17, 2015) of National Board of Higher Mathematics, Department of Atomic Energy, Government of India.
\end{acknowledgement}

\end{document}